%% file: settheories.tex
\documentclass{amsart}
\input{astdecls}
\usepackage[status=draft,author=]{fixme}
\fxusetargetlayout{color}
\title{Comparing material and structural set theories}
\author{Michael Shulman}
\thanks{The author was supported by a National Science Foundation postdoctoral fellowship while writing the original version of this paper.
  While revising it he was supported by the United States Air Force Research Laboratory under agreement number FA9550-15-1-0053.  The U.S. Government is authorized to reproduce and distribute reprints for Governmental purposes notwithstanding any copyright notation thereon.  The views and conclusions contained herein are those of the author and should not be interpreted as necessarily representing the official policies or endorsements, either expressed or implied, of the United States Air Force Research Laboratory, the U.S. Government, or Carnegie Mellon University.\\
\begin{tabular}{cp{0.8\textwidth}}\raisebox{-10px}{\includegraphics[width=44px]{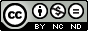}} &
This work is licensed under a \href{https://creativecommons.org/licenses/by-nc-nd/4.0/}{Creative Commons Attribution-NonCommercial-NoDerivatives 4.0 International License.}
\end{tabular}}
\address{University of San Diego, 5998 Alcala Park, San Diego, California, USA}

\begin{document}

\begin{abstract}
  We study elementary theories of well-pointed toposes and pretoposes, regarded as category-theoretic or ``structural'' set theories in the spirit of Lawvere's ``Elementary Theory of the Category of Sets''.
  We consider weak intuitionistic and predicative theories of pretoposes, and we also propose category-theoretic versions of stronger axioms such as unbounded separation, replacement, and collection.
  Finally, we compare all of these theories formally to traditional membership-based or ``material'' set theories, using a version of the classical construction based on internal well-founded relations.
\end{abstract}

\maketitle

\setcounter{tocdepth}{1}
\tableofcontents

\section{Introduction}
\label{sec:introduction}

\subsection{Overview}
\label{sec:overview}

It is well-known that elementary topos theory is equiconsistent with a weak form of set theory.
The most common statement of this relates the theory of a well-pointed topos with a natural numbers object and satisfying the axiom of choice (a.k.a.\ ETCS; see~\cite{lawvere:etcs}) to ``bounded Zermelo'' set theory (BZC).
This is originally due to Cole~\cite{cole:cat-sets} and Mitchell~\cite{mitchell:topoi-sets}; see also Osius~\cite{osius:cat-setth} and for a more recent exposition~\cite[Ch.~VI]{mm:shv-gl}.
Moreover, the proof is very direct: the category of sets in any model of BZC is a model of ETCS, while from any model of ETCS one can construct a model of BZC out of well-founded relations.\footnote{These ``constructions'' can also be expressed as logical translations from one first-order language into another, thereby obtaining an equiconsistency result that does not depend on any ability to ``construct'' new models in the metatheory. See \autoref{rmk:translation}.}

Our goal is to obtain analogous results for stronger and weaker theories in a unified framework.
On the one hand, by removing axioms such as choice, classical logic, and power sets from BZC, we obtain intuitionistic and predicative set theories which correspond to more general toposes and pretoposes.
On the other hand, we propose category-theoretic properties which correspond to set-theoretic axioms such as separation, replacement, and collection.
In particular, we characterize toposes that are the category of sets in a model of ZFC, as did Osius, Lawvere, and also McLarty~\cite{mclarty:catstruct}; but our result generalizes to the intuitionistic case.

In this paper we deal entirely with (pre)toposes which correspond \emph{directly} to set theories, in such a way that the set-theoretic elements of a set $X$ are in bijection with its category-theoretic ``global elements'' $1\to X$.  In particular, any object of such a category must be determined by its global elements, which is to say that the terminal object $1$ is a generator; that is, the category is \emph{well-pointed}.  (We will work with a constructive version of well-pointedness appropriate to intuitionistic logic.)

Of course, there is a further aspect to the classical story: if we drop well-pointedness, we can interpret BZC in any Boolean topos with an \nno\ and choice by using the \emph{internal logic}.
In~\cite{shulman:stacksem}, we will extend the usual internal logic to an interpretation we call the \emph{stack semantics} that contains unbounded quantifiers, enabling all the results of this paper to be carried out ``internally'' in non-well-pointed categories.
Other related ways to deal with unbounded quantifiers include~\cite{jm:ast,abss:long-version}.

\subsection{Some remarks on terminology}
\label{sec:terminology}

While ``set theory'' often refers only to membership-based theories such as ZFC and BZC, clearly ETCS and its ilk are also ``theories about sets'', and Lawvere has observed that ETCS can serve just as well as BZC as a foundation for much mathematics.
Pursuant to this philosophy, we will use the term \emph{material set theory} for theories based on a global membership predicate, such as BZC and ZFC, and \emph{structural set theory} for theories which take functions as fundamental, such as ETCS.

In material set theories, the elements of a set $X$ have an independent identity, apart from being collected together as the elements of $X$.
Frequently, they are also sets themselves.
The adjective ``material'' is a suggestion of Steve Awodey~\cite{awodey:struct-mathlog}; these are also called ``membership-based'' set theories.

In structural set theories, the elements of a set $X$ have no identity independent of $X$, and in particular are not sets themselves; they are merely abstract ``elements'' with which we build mathematical structures.
Moreover, at least in the structural set theories we describe here, it is not meaningful to ask whether two elements of two different sets are equal; only whether two elements of a given ambient set are equal.
Structural set theories are sometimes called ``categorical''\footnote{To avoid conflict with other meanings of ``categorical'', some people use ``categorial'' instead.} set theories; we call them ``structural'' because they are closely aligned with the mathematical philosophy of ``structuralism'' (see e.g.~\cite{mclarty:numbers,awodey:struct}), and because although they are usually formulated in the language of category theory, there is nothing intrinsically necessary about that (see e.g.~\cite{nlab:sear}).

As foundations for mathematics, material and structural set theory each have advantages and niches.
Material set theory is well-adapted to the study of transfinite recursion and related subjects in pure set theory.
However, structural set theory is arguably closer to mathematical practice outside of set theory, where the intricate $\in$-structure of material set theory is irrelevant and objects need only be defined up to isomorphism.\footnote{Indeed, Bourbaki, who enshrined a structuralist philosophy in their ``general theory of structure''~\cite{bourbaki:sets}, had to carefully specify exactly how the superfluous data of material set theory was to be forgotten, in order that all of their structures would be invariant under isomorphism.}
Structural set theory is also more closely related to type theory, and might be considered a halfway-point between a material set-theoretic foundation and a type-theoretic one.
The value of equivalences such as those we study here, then, is that they enable us to use whichever foundation is most appropriate for a given subject, with a canonical way to pass back and forth as necessary.

\subsection{Summary of contents}
\label{sec:organization-paper}

We begin in \S\ref{sec:sets} by listing some common axioms of material set theory, most of which should be familiar.
Then in \S\S\ref{sec:struct-set}--\ref{sec:strong-ax} we develop general aspects of structural set theory.
Little of this appears in the literature; the only papers we are aware of specifically about structural set theory are Lawvere's original~\cite{lawvere:etcs} about ETCS, and the more recent~\cite{palmgren:cetcs} which presents a predicative version thereof (``CETCS''), plus expository works such as~\cite{lr:sets-for-math,leinster:etcs}.
Our treatment is also a little idiosyncratic, due to a desire to remain isomorphism-invariant. 

We begin in \S\ref{sec:struct-set} by setting out the formal language we will use for category theory and recalling the basic notion of a \emph{Heyting pretopos}, which underlies all our structural set theories.
In \S\ref{sec:constr-well-point} we add the further axiom of \emph{well-pointedness}, which makes a Heyting pretopos into a structural set theory.
Since well-pointedness has mostly been studied in a classical setting, we spend some time on the appropriate version thereof for a constructive theory, and its consequences.
Then in \S\ref{sec:sst-as-setth} we rephrase the notion of ``constructively well-pointed Heyting pretopos'' in a way which looks more like a set theory.

In \S\ref{sec:more-sst} we consider three further fairly straightforward axioms---full separation, fullness, and full induction---while in \S\ref{sec:strong-ax} we study collection and replacement \mbox{axioms}, which are rather more subtle.
Replacement-type axioms for ETCS have been proposed by Lawvere~\cite{lawvere:etcs}, Cole~\cite{cole:cat-sets}, Osius~\cite{osius:cat-setth}, and McLarty~\cite{mclarty:catstruct}, but none of these is really suitable intuitionistically.
We begin with collection, which relaxes the hypothesis of replacement by not requiring uniqueness.
This is well-known to be important in weaker set theories, and its structural version is also easier to state.
The structural replacement axiom is surprisingly subtle, since \emph{uniqueness} is asserted only up to unique canonical isomorphism; we give a precise formulation of this and show that it follows from collection.

As we introduce each axiom of structural set theory, we prove that if \bV\ is a model of material set theory satisfying the corresponding material axiom, then the structural axiom is satisfied by its category of sets $\bbSet(\bV)$.
Not until \S\ref{sec:constr-mat} do we turn to the reverse construction, originally due to Cole, Mitchell, and Osius, which produces a model $\bbV(\Set)$ of material set theory starting from a suitable pretopos \Set.
As expected, we then show that if \Set\ satisfies any given structural axiom, $\bbV(\Set)$ satisfies the corresponding material axiom.
In particular, we conclude that the following pairs of theories are equiconsistent:
\begin{blist}
\item Lawvere's ETCS and BZC, as in the classical theorems.
\item ETCS plus our axioms of collection or replacement (which are equivalent over ETCS and imply separation), and classical ZFC.
\item ``Intuitionistic ETCS'' (the theory of a constructively well-pointed topos with a \nno), and ``Intuitionistic BZ'' (containing power sets and infinity).
\item Intuitionistic ETCS plus our axioms of separation and collection, and IZF.
\item The theory of a constructively well-pointed $\Pi$-pretopos with a \nno\ satisfying our axioms of fullness, collection, and extensional well-founded induction, and Aczel's CZF augmented by Mostowski's principle.
\end{blist}

Finally, in \S\ref{sec:conclusion} we consider the extent to which $\bbV(-)$ and $\bbSet(-)$ are inverses.
This depends on foundation-type axioms in material set theory, and on a dual class of axioms for structural set theory that we call ``materialization''.



\subsection{Conventions and Notation}
\label{sec:conventions-notation}

We treat both material and structural set theories as theories in
intuitionistic first-order logic (with a twist in the structural case; see \S\ref{sec:struct-set}).  We regard classical logic as an
axiom schema, $\ph\vee\neg\ph$ for all (or some) formulas $\ph$, which may or
may not be assumed in a theory or satisfied by a model.



All axioms and formulas are of course properly expressed in a formal
first-order language.  However, to make the paper easier to read for
non-logicians, we will usually write such formulas out in English, trusting the
reader to translate them into formal symbols as necessary.  To make it
clear when our mathematical English is to be interpreted as code for a
logical formula, we will often use sans serif font and corner quotes.
Thus, for example, in a structural set theory, \qq{every surjection
  has a section} represents the axiom of choice.

\subsection{Acknowledgements}
\label{sec:acknowledgements}

I would like to thank Steve Awodey, Toby Bartels, \mbox{Peter} Lumsdaine, Colin McLarty, Thomas Streicher, Paul Taylor, and the patrons of the $n$-Category Caf\'e for many helpful conversations and useful comments on early drafts of this paper.
The referee of the preliminary version also gave many helpful suggestions.


\section{Material set theories}
\label{sec:sets}

Our \emph{material set theories} will be theories in one-sorted first
order logic, with a single binary relation symbol \ordin, whose
variables are called \emph{sets}.  (We will briefly mention the
possibility of ``atoms'' which are not sets at the end of this
section.)

\begin{notn}\label{notn:contexts}
  There are two common conventions regarding the treatment of free
  variables in logical formulas.  According to one convention, we
  should write $\ph(x_1,\dots,x_n)$ to indicate that all the free
  variables of \ph\ are among $x_1,\dots,x_n$, and then write
  $\ph(a_1,\dots,a_n)$ for the result of substituting each $a_i$ for
  the corresponding $x_i$.  According to the other convention, in the
  first case we write $\ph$ without notating the free
  variables, and then $\ph[a_1/x_1, \dots, a_n/x_n]$ for the result of
  substitution.

  The former convention has the advantage of an easy and intuitive
  notation for substitution, plus explicit indication of the free
  variables.  The latter convention is cleaner if there are many free
  variables but only a few are being substituted for.

  We choose a convention that borrows something from both standard ones.  By a \textbf{context} we
  mean a finite list $\Gamma = (x_1, \dots, x_n)$ of distinct
  variables.  (This definition will change slightly in the next
  section, where we use a typed logic.)  A formula \ph, written with
  no variables indicated, is said to be \textbf{in context \Gamma} if
  all of its free variables are found in \Gamma.  However, we say a
  formula written as $\ph(x,y)$ (for instance) is \textbf{in context
    \Gamma} if all of its free variables \emph{other} than $x$ and $y$
  are found in \Gamma\ (and $x$ and $y$ do not appear in \Gamma).
  This convention allows us to carry around many free variables with
  unchanging values without needing to notate them constantly, yet
  retains a convenient notation indicating those free variables whose
  values do change.

  In particular, this convention simplifies the statements of many of
  the axioms below.  In all axiom schemas which are parameterized over
  a formula \ph, $\ph(x)$, or $\ph(x,y)$, we assume this formula to
  exist in an ambient context \Gamma, whose variables are implicitly
  universally quantified outside each instance of the schema.
\end{notn}

\begin{defn}
  In material set theory, a \textbf{\ddo-quantifier} is one of the form
  \begin{blist}
  \item $\exists x. (x\in a \wedge \cdots)$, abbreviated $\exists x\in
    a.$; or
  \item $\forall x. (x\in a \imp \cdots)$, abbreviated $\forall x\in
    a.$.
  \end{blist}
  A \textbf{\ddo-formula} (also called \textbf{restricted} or
  \textbf{bounded}) is one with only \ddo-quantifiers.
\end{defn}

The prefix ``\ddo-'' on any axiom schema indicates that the formulas
appearing in it are restricted to be \ddo; we sometimes use the
adjective ``full'' to indicate the lack of any such restriction.

\begin{defn}\label{defn:rels}\ 
  \begin{blist}
  \item We write $\setof{ x | \ph(x)}$ for a set $a$ such that $x\in a
    \Iff \ph(x)$, if such a set exists.
  \item A \textbf{relation} from $a$ to $b$ is a set $r$ of ordered
    pairs $(x,y)$ such that $x\in a$ and $y\in b$, where $(x,y) =
    \{\{x\},\{x,y\}\}$.\footnote{It is straightforward to prove that
      this Kuratowski definition of ordered pairs still has the
      correct behavior intuitionistically; see for
      instance~\cite{aczel:cst}.}
  \item A relation $r$ from $a$ to $b$ is \textbf{entire} (or a
    \emph{many-valued function}) if for all $x\in a$ there exists
    $y\in b$ with $(x,y)\in r$.  If each such $y$ is unique, $r$ is a
    \textbf{function}.
  \item A relation $r$ is \textbf{bi-entire} if both $r$ and $r^o =
    \setof{(y,x)|(x,y)\in r}$ are entire.
  \end{blist}
\end{defn}

The following set of axioms (in intuitionistic logic) we call the
\textbf{core axioms}; they form a subsystem of basically all material
set theories.
\begin{blist}
\item \emph{Extensionality}: \qq{for any $x$ and $y$, $x=y$ if and
    only if for all $z$, $z\in x$ iff $z\in y$}.
\item \emph{Empty set:} \qq{the set $\emptyset = \{\}$ exists}.
\item \emph{Pairing:} \qq{for any $x$ and $y$, the set $\{x,y\}$
    exists}.
\item \emph{Union:} \qq{for any $x$, the set $\setof{z | \exists y\in
      x. (z\in y)}$ exists}.
\item \emph{\ddo-Separation:} For any \ddo-formula $\ph(x)$, \qq{for
    any $a$, $\setof{x\in a | \ph(x)}$ exists}.
\item \emph{Limited \ddo-replacement}: For any \ddo-formula
  $\ph(x,y)$, \qq{for all $a$ and $b$, if for every $x\in a$ there
    exists a unique $y$ with $\ph(x,y)$, and moreover $z\subseteq b$
    for all $z\in y$, then the set $\setof{y | \exists x\in
      a.\ph(x,y)}$ exists}.
\end{blist}

The first five of these are standard, while the last, though very
weak, is unusual.  We include it because we want our core axioms to
imply a well-behaved category of sets, for which purpose the other
five are not quite enough.  This axiom is satisfied in basically all
material set theories, since it is implied both by ordinary
replacement (even \ddo-replacement), and by the existence of power
sets (using \ddo-separation).



We now list some additional common and important axioms of material
set theory, pointing out the most important implications between them.


One major axis along which set theories vary is ``impredicativity,''
expressed by the following sequence of axioms (of increasing
strength):
\begin{blist}
\item \emph{Exponentiation:} \qq{for any $a$ and $b$, the set $b^a$ of
    all functions from $a$ to $b$ exists}.
\item \emph{Fullness:} \qq{for any $a,b$ there exists a set $m$ of
    entire relations from $a$ to $b$ such that for any entire
    relation $f$ from $a$ to $b$, there is a $g\in m$ with
    $g\subseteq f$}.
\item \emph{Power set:} \qq{for any set $x$, the set $P x = \setof{y |
      y \subseteq x}$ exists}.
\end{blist}
Each of these is strictly stronger than the preceding one, but in the
presence of either of the following:
\begin{blist}
\item \emph{\ddo-classical logic:} for any \ddo-formula \ph, we have
  $\ph\join\neg\ph$.
\item \emph{Full classical logic:} for any formula \ph, we have
  $\ph\join\neg\ph$.
\end{blist}
the hierarchy of predicativity collapses, since power sets can then be
obtained (using limited \ddo-replacement) from sets of functions into a
2-element set.  Hence the combination of classical logic with exponentiation is also impredicative.  Also
considered impredicative is:
\begin{blist}
\item \emph{(Full) Separation:} For any formula $\ph(x)$, \qq{for any
    $a$, $\setof{x\in a | \ph(x)}$ exists}.
\end{blist}
Full separation and \ddo-classical logic together imply full classical
logic, since we can form $a = \setof{\emptyset|\ph}\subseteq
\setof{\emptyset}$, and then $(\emptyset\in a)\join (\emptyset\not\in
a)$ is an instance of \ddo-classical logic that implies
$\ph\join\neg\ph$.
A similar argument shows that \ddo-classical logic is equivalent to the assertion that every subset has a complement.


We also have the collection and replacement axioms.
Our collection axiom is often called ``Strong Collection''.
\begin{blist}
\item \emph{Collection:} for any formula $\ph(x,y)$, \qq{for any $a$,
    if for all $x\in a$ there is a $y$ with $\ph(x,y)$, then there is
    a $b$ such that for all $x\in a$, there is a $y\in b$ with
    $\ph(x,y)$, and for all $y\in b$, there is a $x\in a$ with
    $\ph(x,y)$}.
\item \emph{Replacement:} for any formula $\ph(x,y)$, \qq{for any $a$,
    if for every $x\in a$ there exists a unique $y$ with $\ph(x,y)$,
    then the set $\setof{y | \exists x\in a.\ph(x,y)}$ exists}.
\end{blist}
With an evident extension of \autoref{defn:rels} to formulas, we can state
the collection axiom as ``if $\ph$ is entire on $a$, then there exists
$b$ such that $\ph$ is bi-entire from $a$ to $b$.''  Collection also
evidently implies replacement.  Classical ZF set theory is often
defined using replacement rather than collection, but in the presence
of the other axioms of ZF, collection can be proven.  (Replacement and
full classical logic together also imply full separation.)  This is no
longer true in intuitionistic logic, though, or in classical logic
without the axiom of foundation, and in practice the extra strength of
collection is sometimes important.  There are also other versions of
collection, all of which are equivalent in the presence of full
separation, but the one we have chosen seems most appropriate in its
absence.


We consider only the version of the axiom of infinity which asserts
the existence of the set $\omega$ of finite von Neumann ordinals.
\begin{blist}
\item \emph{(von Neumann) Infinity:} \qq{there exists a set $\omega$
    such that $\emptyset\in\omega$, if $x\in \omega$ then $x\cup
    \{x\}\in \omega$, and if $z$ is any other set such that $\emptyset
    \in z$ and if $x\in z$ then $x\cup \{x\}\in z$, then
    $\omega\subseteq z$}.
\end{blist}
We can omit the final clause in this axiom in the presence of power sets or full separation, since it can be recovered by replacing $\omega$ with its smallest subset containing $\emptyset$ and closed under successor $x\mapsto x\cup \{x\}$.
But otherwise we need to include this clause explicitly, in order to prove that
mathematical induction holds for \ddo-formulas.  For a stronger
version of induction, we need a separate axiom.
\begin{blist}
\item \emph{(Full) Induction:} for any formula $\ph(x)$, \qq{if $\ph(\emptyset)$
    and $\ph(x) \imp \ph(x\cup \{x\})$ for any $x\in \omega$, then
    $\ph(x)$ for all $x\in \omega$}.
\end{blist}
Of course, infinity plus full separation implies full induction.

There are also various axioms of choice; we will consider only the following two.
\begin{blist}
\item \emph{Choice:} \qq{for any set $a$, if for all $x\in a$ there
    exists a $y\in x$, then there exists a function $f$ from $a$ to
    $\bigcup a$ such that $f(x)\in x$ for all $x\in a$}.
\item \emph{Presentation:} \qq{For every set $a$, there exists a
    surjection $b\to a$, where $b$ has the property that any
    surjection $c\to b$ has a section.}
\end{blist}
By Diaconescu's argument~\cite{diaconescu:ac}, choice and power set
together imply \ddo-classical logic.  A set $b$ as in the presentation
axiom is called \emph{projective}; thus the presentation axiom is also
called ``EPSets'' (Enough Projective Sets) or ``COSHEP'' (the Category
Of Sets Has Enough Projectives).  It implies countable dependent
choice.



Finally, we have the family of ``foundation-type'' axioms.  Recall
that a set $x$ is called \emph{transitive} if $z\in y\in x$ implies
$z\in x$.
A set $X$ with a binary relation $\prec$ is called \emph{well-founded}
if whenever $Y\subseteq X$ is \emph{inductive}, in the sense that
$\forall y\prec x. (y\in Y)$ implies $x\in Y$, then in fact $Y=X$.  It
is \emph{extensional} if $\forall z\in X.(z\prec x\Iff z\prec y)$
implies $x=y$.
\begin{blist}
\item \emph{Set-induction:} for any formula $\ph(x)$, \qq{if for any
    set $y$, $\ph(x)$ for all $x\in y$ implies $\ph(y)$, then $\ph(x)$
    for any set $x$}.
\item \emph{Foundation}, a.k.a.\ \emph{\ddo-set-induction:} like
  set-induction, but only for \ddo-formulas.
\item \emph{Transitive closures:} \qq{every set is a subset of a
    smallest transitive set}.
\item \emph{Mostowski's principle:} \qq{every well-founded extensional
    relation is isomorphic to a transitive set equipped with the
    relation $\in$}.
\end{blist}
The usual ``classical'' formulation of foundation is
\begin{blist}
\item \emph{Regularity:} \qq{if $x$ is nonempty, there is some $y\in
    x$ such that $x\cap y = \emptyset$}.
\end{blist}
In the presence of transitive closures, regularity is equivalent to
the conjunction of \ddo-set-induction and \ddo-classical logic.  Thus,
in an intuitionistic theory, \ddo-set-induction is preferable.  Of
course, \ddo-set-induction follows from full set-induction, while the
converse is true if we have full separation.  Set-induction also
implies full ordinary induction, since $<$ and $\in$ are the same for von
Neumann ordinals.

Transitive closures can be constructed either from set-induction and
replacement (see~\cite{aczel:cst}), or from infinity, ordinary (full)
induction, and replacement.  Mostowski's original proof shows that his
principle follows from replacement and full separation.  Transitive
closures and Mostowski's principle are not usually considered
``foundation'' axioms, but we group them thusly for several reasons:
\begin{blist}
\item The conjunction of transitive closures, Mostowski's principle,
  and foundation is equivalent to the single statement ``to give a
  \emph{set} is the same as to give a well-founded extensional
  relation with a specified element.''  One could argue that this sums
  up the objects of study of well-founded material set theory.
\item As we will see in \S\ref{sec:constr-mat}, it follows that the (well-founded)
  material set theories that can be constructed from structural ones
  are precisely those which satisfy these three axioms.
\item Finally, the various axioms of \emph{anti-foundation} are also
  naturally stated as ``to give a set is the same as to give a
  \emph{(blank)} relation with a specified element,'' where
  \emph{(blank)} denotes some stronger type of extensionality (there
  are several choices; see~\cite{aczel:afa}).  Thus, these axioms
  naturally include analogues not just of foundation, but also of
  transitive closures and Mostowski's principle.
\end{blist}
Similar groupings can be found in, for
instance,~\cite{abss:long-version} and~\cite{mathias:str-maclane}.

Some particular material set theories with names include the
following.  Note that in this context, the prefix \emph{constructive}
is generally reserved for ``predicative'' theories (including,
potentially, exponentiation and fullness, but never power sets or full
separation) while \emph{intuitionistic} refers to ``impredicative''
(but still non-classical) theories which may have power sets or full
separation.  None of these theories include our axiom of limited
\ddo-replacement explicitly, but all include either power sets or a
stronger version of replacement, both of which imply it.
\begin{blist}
\item The standard \textbf{Zermelo-Fraenkel set theory with Choice
    (ZFC)} includes \emph{all} the axioms mentioned above, although it
  suffices to assume the core axioms together with full classical
  logic, power sets, replacement, infinity, foundation, and choice.
\item By \textbf{Zermelo set theory (Z)} one usually means the core
  axioms together with full classical logic, power sets, full
  separation, and infinity.
\item The theory \textbf{ZBQC} of~\cite{maclane:formandfunction},
  called \textbf{RZC} in~\cite[VI.10]{mm:shv-gl} and \textbf{BZC} elsewhere, consists of the core
  axioms together with full classical logic, power sets, infinity,
  choice, and regularity.
\item By \textbf{Intuitionistic BZ} we will mean the core axioms together with power sets, infinity, and foundation.
\item The \textbf{Mostowski set theory (MOST)}
  of~\cite{mathias:str-maclane} consists of \textbf{ZBQC} together
  with transitive closures and Mostowski's principle.
\item By \textbf{Kripke-Platek set theory (KP)} is usually meant the
  core axioms together with full classical logic, a form of
  \ddo-collection (sufficient to imply our limited \ddo-replacement),
  and set-induction for some class of formulas.
\item The theory \textbf{ZFC$-$} of~\cite{ghj:zfcminus} consists of the
  core axioms together with full classical logic, infinity,
  foundation, choice, and replacement (the ``sufficient axioms''
  listed above for ZFC, with power sets removed), while their
  \textbf{ZFC$^-$} has collection instead of replacement.  (Actually,
  they use a version of choice that is stronger than ours in the
  absence of power sets, but is unsuitable intuitionistically.)
\item Aczel's \textbf{Constructive Zermelo-Fraenkel (CZF)} consists of
  the core axioms together with fullness, collection, infinity, and
  set-induction; see~\cite{aczel:cst}.  His \textbf{CZF$_0$} consists
  of the core axioms, replacement, infinity, and induction.
\item The usual meaning of \textbf{Intuitionistic Zermelo-Fraenkel
    (IZF)} is the core axioms together with power sets, full
  separation, collection, infinity, and set-induction.  This is the
  strongest theory that can be built from the above axioms without any
  form of classical logic or choice.
\item The \textbf{Rudimentary Set Theory (RST)}
  of~\cite{vdbm:pred-i-exact} consists of the core axioms together
  with collection and set-induction.
\end{blist}

So far we have considered only theories of pure sets, i.e.\ in which
everything is a set.  If we want to also allow ``atoms'' or
``urelements'' which are not sets, we can modify any of the above
theories by adding a predicate \qq{$x$ is a set}, whose negation is
read as \qq{$x$ is an atom}.  We must then add the following axiom:
\begin{blist}
\item \emph{Only sets have elements:} \qq{if $x\in y$, then $y$ is a set}.
\end{blist}
and modify the extensionality axiom so that it applies only to sets:
\begin{blist}
\item \emph{Extensionality for sets:} \qq{for any sets $x$ and $y$,
    $x=y$ if and only if for all $z$, $z\in x$ iff $z\in y$}.
\end{blist}
We may choose whether or not to add the following axioms:
\begin{blist}
\item \emph{Decidability of sethood:} \qq{every $x$ is either a set or
    an atom}.
\item \emph{Set of atoms:} \qq{the set $\setof{x | x \text{ is an
        atom}}$ exists}.
\end{blist}
as well as whether to allow the predicate \qq{$x$ is a set} to appear
in the formulas that parameterize the axiom schemas.  For instance:
\begin{blist}
\item The theory \textbf{IZFA} of~\cite{jm:ast} consists of IZF, as
  above, with atoms and decidable sethood, and does allow sethood in
  axiom schemas.
\item If \textbf{ZFA} has a standard meaning, it is probably
  IZFA plus full classical logic.
\item The \textbf{Basic Constructive Set Theory (BCST)}
  of~\cite{aw:pred-ast,afw:ast-ideals} consists of the core axioms
  together with replacement, with atoms.  Their \textbf{Constructive
    Set Theory (CST)} adds exponentials to this, and the \textbf{Basic
    Intuitionistic Set Theory (BIST)}
  of~\cite{af:ast-classes,afw:ast-ideals,abss:bulletin-announcement,abss:long-version}
  adds power sets and a non-von-Neumann version of infinity.  These
  theories do not assume decidability of sethood or a set of atoms,
  and do not allow sethood in axiom schemas.
\item The theory \textbf{BIZFA} of~\cite{abss:long-version} consists
  of BIST, as above, together with decidable sethood, foundation,
  transitive closures, and Mostowski's principle (suitably modified
  for the existence of atoms), and does allow sethood in axiom
  schemas.
\end{blist}

\section{Heyting pretoposes}
\label{sec:struct-set}

We now move on to structural set theories, beginning with the underlying logic.
The theory of a category can, of course, be formulated in ordinary first-order logic, but this is unsatisfactory to us because it allows us to discuss equality of objects.
%
However, we cannot simply remove the equality predicate on objects from the ordinary first-order theory, because sometimes we \emph{do} need to know that two objects are the same.
For example, to compose two arrows $f$ and $g$, we need to know that the source of $g$ is the same as the target of $f$.
The solution (c.f.~\cite{freyd:invar-eqv,blanc:eqv-log,makkai:folds,makkai:comparing}) is to use a language of \emph{dependent types}.
We give here a somewhat informal description in an attempt to avoid the complicated precise syntax of dependent type theory, but the reader familiar with DTT should easily be able to expand our description into a formal one.

\begin{defn}\label{defn:lang-cat}
  The \textbf{language of categories} consists of the following.
  \killspacingtrue
  \begin{enumerate}
  \item There is a type of \emph{objects}.
    By an \emph{object-variable} we mean a variable that has been assigned to the type of objects.
  \item Each object-variable is, in particular, an \emph{object-term}.
  \item For object-terms $X$ and $Y$, there is a type $X\to Y$ of \emph{arrows from $X$ to $Y$}.
    By an \emph{arrow-variable} we mean a variable that has been assigned to some type of arrows (which entails specifying the object-terms $X$ and $Y$).
  \item Every arrow-variable is, in particular, an \emph{arrow-term}.
    Like arrow-variables, every arrow-term is assigned to a particular type of arrows.
  \item For each object-term $X$, there is an arrow-term $1_X\colon X\to X$.
  \item For any object-terms $X$, $Y$, and $Z$ and each pair of arrow-terms $f\maps X\to Y$ and $g\maps Y\to Z$, there is an arrow-term $(g\circ f)\maps X\to Z$.
  \item For any arrow-terms $f,g\colon X\to Y$, there is an atomic formula $(f=g)$.
  \end{enumerate}
  \killspacingfalse
  Non-atomic formulas are built up from atomic ones in the usual way, using connectives $\top$, $\bot$, $\meet$, $\join$, $\imp$, $\neg$, and quantifiers $\im$ and $\coim$.
  There are of course two types of quantifiers, for object-variables and for arrow-variables; the former is restricted in that we may not quantify over an object-variable if it appears in the type of any \emph{free} arrow-variable.
\end{defn}

Of course, this theory uses very little of the full machinery of DTT, and in particular it may be said to live in the ``first-order'' fragment of DTT.
In this way it is closely related to the logic FOLDS studied in~\cite{makkai:folds}, although we are allowing some term constructors in addition to relations.

\begin{notn}\label{notn:struct-context}
  We will use the same notational convention for free variables as in \autoref{notn:contexts}, but adapted appropriately for dependent type theory.
  In particular, by a \textbf{context} in this language, we mean a finite (ordered) list \Gm\ of typed variables, with the property that whenever an arrow-variable $f\colon X\to Y$ occurs in \Gm, each of $X$ and $Y$ is an object-variable occurring in \Gm\ prior to $f$.
  If $\Theta$ is a further list of typed variables which is an \textbf{extension} of \Gm, in the sense that the concatenation $(\Gamma,\Theta)$ is a context, then we say that a formula written $\ph(\Theta)$ is \textbf{in context \Gm} if all its free variables are in $(\Gamma,\Theta)$.

  Thus, for instance, $\Gm=(X,Y,f\colon X\to Y)$ is a context, and if $g\colon X\to Y$ is a further arrow-variable, then a formula whose free variables are among $X,Y,f,g$ may be denoted $\ph(g)$ in context \Gm.
\end{notn}

\begin{defn}
  Let \bA\ be a category, \Gm\ a context, and let $\rho$ be a well-typed assignment of the variables in \Gm\ to objects and morphisms of \bA.
  We denote this by $\rho\colon \Gm \to\bA$.
  Then the \textbf{satisfaction} relation $\bA \ss_\rho \ph$ is defined as usual, by induction on the structure of \ph.
  If $\rho$ is obvious from context, we omit it from the notation.
\end{defn}

The main thing to note about this language is that, as in FOLDS, the only atomic formulas are equalities between parallel arrows; the language does not allow us to even discuss whether two objects are equal.
(However, we do still assume the usual rules of equality for equalities of arrows, such as substitution, reflexivity, symmetry, and transitivity.)
This implies that satisfaction is isomorphism- and equivalence-invariant, in the following sense.

First, suppose $\rho,\rho'\colon\Gm\to\bA$.
By definition, an \textbf{isomorphism} $\rho\cong \rho'$ consists of an isomorphism $\rho(X) \cong \rho'(X)$ for each object-variable $X\in\Gm$, such that for every arrow-variable $f\colon X\to Y$ in \Gm, the following square commutes:
\[\xymatrix@-.5pc{ \rho(X) \ar[r]^{\rho(f)} \ar[d]_{\cong} & \rho(Y)
  \ar[d]^{\cong}\\
  \rho'(X) \ar[r]_{\rho'(f)} & \rho'(Y). } \]

\begin{lem}[Isomorphism-invariance of truth]\label{thm:isoinvar-truth}
  If $\rho\cong\rho'$ as above, and \ph\ is a formula in context \Gm, then $\bA\ss_\rho \ph$ if and only if $\bA\ss_{\rho'}\ph$.\qed
\end{lem}

Second, if $F\colon \bA\to\bB$ is a functor and $\rho\colon \Gm\to \bA$, we write $F\circ\rho\colon \Gm\to\bB$ for the assignment determined by $(F\circ\rho)(X) = F(\rho(X))$ and $(F\circ\rho)(f) = F(\rho(f))$.

\begin{lem}[Equivalence-invariance of truth]
  If \ph\ is a formula in context \Gm, and the functor $F$, as above, is fully faithful and essentially surjective, then $\bA\ss_{\rho}\ph$ if and only if $\bB \ss_{F\circ\rho} \ph$.\qed
\end{lem}

Both of these lemmas are proven by straightforward induction on the structure of \ph.
(In fact, the dependently typed theory of categories characterizes exactly those properties of categories which are invariant under equivalence in this sense; see~\cite{blanc:eqv-log,freyd:invar-eqv,makkai:folds}.)

We now list the basic category-theoretic properties we will consider.
First of all, we must have a category, so we include:
\begin{blist}
\item \emph{Identity:} \qq{for all $X,Y$, and $f\colon X\to Y$, we have $1_Y\circ f = f = f\circ 1_X$}.
\item \emph{Associativity:} \qq{for all $X,Y,Z,W$ and $f\colon X\to Y$, $g\colon Y\to Z$, and $h\colon Z\to W$, we have $h\circ (g\circ f)=(h\circ g)\circ f$}.
\end{blist}
Recall that a \textbf{Heyting category} is a category satisfying the following axioms.
\begin{blist}
\item \emph{Finite limits:} \qq{there exists a terminal object, binary products, and equalizers}.
\item \emph{Regularity:} \qq{every morphism factors as a regular epi followed by a mono, and regular epis are stable under pullback}.
\item \emph{Coherency:} \qq{finite unions of subobjects exist and are stable under pullback}.
\item \emph{Dual images:} \qq{for any $f\maps X\to Y$, the pullback functor $f^*\maps \Sub(Y)\to\Sub(X)$ has a right adjoint $\coim_f$}.
\end{blist}
Of course, a statement such as \qq{binary products exist} actually means
\begin{quote}
  \qq{for all $X,Y$, there exists $W$ and $p\colon W\to X$, $q\colon W\to Y$ such that for any $T$ and $a\colon T\to X$, $b\colon T\to Y$, there exists $h\colon T\to W$ such that $a = p \circ h$ and $b = q\circ h$, and for any $k\colon T\to W$ such that $a = p \circ k$ and $b = q\circ k$, we have $h=k$}.
\end{quote}
The notation $\Sub(X)$ denotes the poset of subobjects of $X$, i.e.\ of monomorphisms into $X$ modulo isomorphism.
In order to remain first-order, the axiom of dual images should be interpreted as asserting the existence of each individual value of $\coim_f$ with its universal property, i.e.
\begin{quote}
  \qq{for any $f\colon X\to Y$, and any monomorphism $m\colon A\to X$, there exists a monomorphism $n\colon B \to Y$ such that for any monomorphism $q\colon C\to Y$, we have that $q$ factors through $n$ if and only if there exists a pullback of $q$ along $f$ which factors through $m$}.
\end{quote}
A Heyting category is called a \textbf{Heyting pretopos} if it also satisfies:
\begin{blist}
\item \emph{Positivity/extensivity:} \qq{there exist binary coproducts which are disjoint and pullback-stable}.
\item \emph{Exactness:} \qq{every equivalence relation is a kernel}.
\end{blist}
It is a \textbf{\Pi-pretopos} if it also satisfies:
\begin{blist}
\item \emph{Local cartesian closure:} \qq{for every arrow $f\maps X\to Y$, the pullback functor $f^*\maps \bS/Y\to\bS/X$ has a right adjoint $\Pi_f$}
\end{blist}
Finally, if it additionally satisfies:
\begin{blist}
\item \emph{Power objects:} \qq{every object $X$ has a power object $P X$}
\end{blist}
(which, together with finite limits, implies all the other axioms of a \Pi-pretopos) it is called an \textbf{(elementary) topos}.
See, e.g.,~\cite[\S A.1]{ptj:elephant} for more details.

In a pretopos, every mono and every epi is regular, and thus it is \emph{balanced} (every monic epic is an isomorphism).
Furthermore, the self-indexing of any pretopos is a stack for its coherent topology.
Of course, first-order versions of these statements can be proven, in the language of categories, from the above axioms.


\begin{rmk}
  Since all of our axioms are formulated in first-order logic, all the categorical operations such as finite limits, exponentials, etc.\ are expressed as pure existence statements.
  Strictly speaking this means that notation such as ``$A\times B$'' or ``$X^Y$'' is an abuse, since we have no product or exponential \emph{operations} (and we are not assuming any axiom of choice allowing us to construct such operations).
  However, such abuses are standard in mathematics and cause no problems; formally they are little different from the notation $\setof{x | \ph(x)}$ in material set theory, whose separation axioms are also formally expressed as pure existence statements.
  In fact, we could make this precise by extending our language to include such operations on terms and proving formally that it can still be interpreted in categories satisfying only the pure existence axioms (thereby moving in the direction of type theory), but we leave this to the interested reader.
\end{rmk}

Some other important axioms are the following:
\begin{blist}
\item \emph{Booleanness:} \qq{every subobject has a complement}.
\item \emph{Full classical logic:} for any formula \ph, we have $\ph\join\neg\ph$.
\item \emph{Infinity:} \qq{there exists a natural number object (\nno)}.
\item \emph{Choice:} \qq{every regular  epimorphism splits}.
\item \emph{Enough projectives:} \qq{Every object admits a regular epimorphism from an object which is projective with respect to regular epimorphisms.}
\end{blist}
(In the absence of cartesian closure, the definition of an \nno\ should be taken with arbitrary parameters; see~\cite[A2.5.3]{ptj:elephant}.
On the other hand, since every epi is regular in a pretopos, we can omit any occurrences of the adjective ``regular''.)

Each of the above axioms (beyond those of a Heyting pretopos) is, more or less obviously, a counterpart of some axiom of material set theory.
This is perhaps least obvious in the case of Booleanness, which we regard as a structural counterpart of \ddo-classical logic.
However, in \S\ref{sec:sets} we observed that in material set theory, \ddo-classical logic is equivalent to the assertion that every sub\emph{set} has a complement.
See also \autoref{thm:wpt-bool} and \autoref{thm:boolean}.

We end this section by observing that each material axiom implies its corresponding structural axiom for the category of sets.

\begin{thm}\label{thm:iz-topos}
  If \bV\ satisfies the core axioms of material set theory, then the sets and functions in \bV\ form a Heyting pretopos $\bbSet(\bV)$.
  Moreover:
  \killspacingtrue
  \begin{enumerate}
  \item If \bV\ satisfies exponentiation, then $\bbSet(\bV)$ is locally cartesian closed.\label{item:iz-topos-lcc}
  \item If \bV\ satisfies the power set axiom, then $\bbSet(\bV)$ is a topos.
  \item If \bV\ satisfies \ddo-classical logic, then $\bbSet(\bV)$ is Boolean.
  \item If \bV\ satisfies full classical logic, so does $\bbSet(\bV)$.
  \item If \bV\ satisfies infinity and exponentials, then $\bbSet(\bV)$ has a \nno.\label{item:iz-topos-nno}
  \item If \bV\ satisfies the presentation axiom, then $\bbSet(\bV)$ has enough projectives.
  \item If \bV\ satisfies the axiom of choice, then so does $\bbSet(\bV)$.
  \end{enumerate}
  \killspacingfalse
\end{thm}
\begin{proof}
  We first show that $\bbSet(\bV)$ is a category.
  Using pairing and limited \ddo-replacement, we can form, for any set $a$ and any $x\in a$, the set $p_{a,x}$ of all pairs $\{x,y\}$ for $y\in a$.
  Again using limited \ddo-replacement, we can form the set of all the $p_{a,x}$ for $x\in a$, and take its union, thereby obtaining the set of all pairs $\{x,y\}$ for $x,y\in a$.
  Applying this construction twice to $a\cup b = \bigcup\{a,b\}$, we can then use \ddo-separation to cut out the set $a\times b$ of Kuratowski ordered pairs $(x,y) = \{\{x\},\{x,y\}\}$ where $x\in a$ and $y\in b$.
  With this in hand we can define function composition and identities using \ddo-separation.

  Finite limits are straightforward: we already have products, $\{\emptyset\}$ is a terminal object, and \ddo-separation supplies equalizers.
  The construction of pullback-stable images, unions, and dual images is also easy from \ddo-separation.
  A stable and disjoint coproduct $a+b$ can as usual be obtained as a subset of $\{0,1\}\times (a\cup b)$.
  And if $r\subseteq a\times a$ is an equivalence relation, \ddo-separation supplies the equivalence class of any $x\in a$, and limited \ddo-replacement then supplies the set of all such equivalence classes.

  Exponentiation clearly implies cartesian closedness.
  For local cartesian closedness, suppose given $f\maps a\to b$ and $g\maps x\to a$.
  Then the fiber of $\Pi_f(g)\to b$ over $j\in b$ should be the set of all functions $s\colon f\inv(j) \to x$ such that $g\circ s = 1_{f\inv(j)}$; this can be cut out of $x^{f\inv(j)}$ by \ddo-separation.
  Finally, the entire set $\Pi_f(g)$ can be constructed from these and limited \ddo-replacement, since each element of each fiber is a subset of $a\times x$.

  The relationship between power sets and power objects is analogous.
  We have just recalled that \ddo-classical logic implies every subset has a complement, and the implication for full classical logic is evident.

  Now suppose that \bV\ satisfies infinity and exponentials, and let $\omega$ be as in the axiom of infinity; we define $0\in \omega$ and $s\colon \omega\to\omega$ in the obvious way.
  Suppose given $A\xto{g} B \xleftarrow{t} B$; we must construct a unique function $f\colon A\times \omega \to B$ such that $f(1_A\times 0) = g \pi_1$ and $f(1_A\times s) = t f$.
  Let $R = \setof{ (a,b)\in \omega\times\omega | a\in b}$ with projection $\pi_2\colon R\to \omega$, and let $X$ be the exponential $(B\times \omega \to \omega)^{(A\times R \to \omega)}$ in $\Set/\omega$; thus an element of $X$ is equivalent to an $n\in \omega$ together with a function $f\colon A \times \setof{ m | m \in n } \to B$.
  Using \ddo-separation, we have the subset $Y\subseteq X$ of those $f\in X$ such that $f(a,0) = g(a)$ and $f(a,s(m)) = t(f(a,m))$ whenever $s(m)\in n$.
  We then prove by induction that for all $n\in\omega$ there exists an $f\in Y$ with $n\in \mathrm{dom}(f)$ and that for any two such $f,f'$ we have $f(n)=f'(n)$.
  The union of $Y$ is then the desired function.

  The presentation axiom, of course, is literally the statement that $\bbSet (\bV)$ has enough projectives.
  And finally, if \bV\ satisfies choice, then from any surjection $p\colon e \to b$ we can construct, using limited \ddo-replacement, the set $\setof{ p^{-1}(x) | x\in b}$, and applying the material axiom of choice to this gives a section of $p$.
\end{proof}

\begin{rmk}\label{rmk:translation}
  We have stated \autoref{thm:iz-topos} in the form ``given a model of material set theory, we can construct a category with certain properties'' because we find it easiest to conceptualize in that way.
  Of course, this requires the model of material set theory to exist in some ambient world in which we can perform this construction.

  A more elementary way to interpret this ``construction'' is as giving a way to translate a formula \ph\ in the language of category theory into a formula in the language of material set theory, by simply writing down what the assertion ``\ph\ holds in the category of sets'' would mean in terms of the basic notions of set theory.
  The content of the theorem then becomes that if we assume certain axioms of material set theory, then the \emph{translations of} certain category-theoretic axioms are provable---and therefore, so is any consequence of those axioms.
  (This is equivalent to performing the above model-theoretic construction with the generic model of a material set theory in its syntactic category.) In this way we can obtain a relative consistency result with the weakest possible metatheoretic assumptions.
  This same sort of interpretation will also be possible for the reverse construction in \S\ref{sec:constr-mat}.
\end{rmk}

\begin{rmk}
  The proof of \autoref{thm:iz-topos} is unaffected by the possible presence of atoms in \bV.
  This will remain true for the theorems in later sections giving similar correspondences for the other axioms.
\end{rmk}


\section{Constructive well-pointedness}
\label{sec:constr-well-point}

In \S\ref{sec:struct-set} we considered only standard category-theoretic properties, but not every Heyting pretopos deserves to be called a model of structural set theory.
The distinguishing characteristic of a \emph{set}, as opposed to an object of some more general category, is that a set is determined by its elements and nothing more.
This is expressed by the following property.

\begin{defn}\label{def:wpt}
  A Heyting category \bS\ is \textbf{constructively well-pointed} if it satisfies the following (in which $1$ refers to a terminal object).
  \killspacingtrue
  \begin{enumerate}[(a)]
  \item If $m\maps A\mono X$ is monic and every $1\too[x] X$ factors through $m$, then $m$ is an isomorphism \emph{(1 is a strong generator)}.\label{item:wpt-strgen}
  \item Every regular epimorphism $X\epi 1$ splits \emph{(1 is projective)}.\label{item:wpt-proj}
  \item If $1$ is expressed as the union of two subobjects $1 = A\cup B$, then either $A$ or $B$ must be isomorphic to $1$ \emph{(1 is indecomposable)}.\label{item:wpt-indec}
  \item There does not exist a map $1\to 0$ \emph{(1 is nonempty)}.\label{item:wpt-nondeg}
  \end{enumerate}
  \killspacingfalse
\end{defn}

\begin{rmk}\label{rmk:nonelem-wpt}
  Of course, since regular epis are stable under pullback, $1$ is projective if and only if for any regular epi $Y\xepi{p} X$, every global element $1\to X$ lifts to $Y$.
  Likewise, since unions are stable under pullback, $1$ is indecomposable if and only if whenever $X = A\cup B$, every global element $1\to X$ factors through either $A$ or $B$.
  In particular, if \bS\ is locally small in some external set theory, then \bS\ is constructively well-pointed if and only if $\bS(1,-)\maps \bS\to\Set$ is a conservative coherent functor.
\end{rmk}

We immediately record the following.

\begin{prop}\label{thm:set-wpt}
  If \bV\ satisfies the core axioms of material set theory, then $\bbSet(\bV)$ is constructively well-pointed.
\end{prop}
\begin{proof}
  Functions $1\to X$ in $\bbSet(\bV)$ are in canonical correspondence with elements of $X$, and a function is monic just when it is injective.
  Thus, if $m$ is monic and every map from $1$ factors through it, then it must be bijective, and hence an isomorphism; thus $1$ is a strong generator.
  Next, the epics in $\bbSet(\bV)$ are the surjections, and if $X\xepi{p} 1$ is a surjection, then $X$ must be inhabited, hence $p$ splits.
  And if $1 = A\cup B$, then the unique element of $1$ must be in either $A$ or $B$ by definition of unions, hence either $A$ or $B$ must be inhabited.
  Finally, $\emptyset$ has no elements, so $1$ is nonempty.
\end{proof}

Recall that classically, a topos is said to be \emph{well-pointed} if $1$ is a generator (i.e. for $f,g\maps X\toto Y$, having $f x=g x$ for all $1\too[x] X$ implies $f=g$) and there is no map $1\to 0$ (1 is nonempty).
Thus, any constructively well-pointed topos is well-pointed in the classical sense.
Conversely, the following is well-known (see, for instance,~\cite[VI.1 and VI.10]{mm:shv-gl}).

\begin{lem}\label{thm:wpt-bool}
  Let \bS\ be a Heyting category satisfying full classical logic, and in which $1$ is a nonempty strong generator.
  Then \bS\ is Boolean and constructively well-pointed.
\end{lem}
\begin{proof}
  For any object $X$, either $X$ is initial or it isn't.
  If it isn't, then there must be a global element $1\to X$, since otherwise the monic $0\mono X$ would be an isomorphism (since $1$ is a strong generator).
  Therefore, if $X$ is not initial, then $X\to 1$ is split epic.
  Now suppose that $X$ is any object such that $X\to 1$ is regular epic; since $0\to 1$ is not regular epic, it follows that $X$ is not initial, and so $X\to 1$ must in fact be split epic; thus $1$ is projective.
  Also, if $1 = A\cup B$, then $A$ and $B$ cannot both be initial, so one of them has a global element; thus $1$ is indecomposable.

  Now let $A\mono X$ be a subobject and assume that $A\cup \neg A$ is not all of $X$.
  Then since $1$ is a strong generator, there is a $1\too[x] X$ not factoring through $A\cup \neg A$.
  Let $V = x^*(A\cup \neg A)$.
  Then $V$ is a subobject of $1$.
  If $V$ is not initial, then it has a global element and hence is all of $1$, so $x$ factors through $A\cup \neg A$.
  But if $V$ is initial, then $x^*(A)$ must also be initial, which implies that $x$ factors through $\neg A$.
  This is a contradiction, so $A\cup \neg A = X$; hence \bS\ is Boolean.
\end{proof}


\begin{cor}
  Let \bS\ be a topos satisfying full classical logic, which is well-pointed in the classical sense.  Then \bS\ is Boolean and constructively well-pointed.
\end{cor}
\begin{proof}
  In a topos, any generator is a strong generator.
\end{proof}

Thus, in the presence of full classical logic, our notion of constructive well-pointedness is equivalent to the usual notion of well-pointedness.
However, \autoref{thm:set-wpt} shows that a constructively well-pointed topos not satisfying full classical logic need not be Boolean.
One can also construct examples showing that in the absence of full classical logic, a Boolean topos in which $1$ is a nonempty generator need not be \emph{constructively} well-pointed.
It is true, however, even intuitionistically, that if $1$ is projective, indecomposable, and nonempty in a Boolean topos, it must also be a generator; see~\cite[V.3]{awodey:thesis}.

\begin{rmk}
Other authors have also recognized the importance of explicitly assuming projectivity and indecomposability of $1$ in an intuitionistic context.
In~\cite{awodey:thesis} a topos in which $1$ is projective, indecomposable, and nonempty was called \emph{hyper\-local} (but in~\cite{ptj:pp-bag-hloc,ptj:elephant} that word is used for a stronger, non-elementary, property).
And in~\cite{palmgren:cetcs}, indecomposability of $1$ is assumed explicitly, while projectivity of $1$ is deduced from a factorization axiom.
It is also worth noting that in an ambient classical logic, a Grothendieck topos is \emph{local}, in the geometric sense that its global sections functor $\bS(1,-):\bS \to \Set$ has a right adjoint, if and only if $1\in\bS$ is projective, indecomposable, and nonempty.
See~\cite{nlab:localgm}.
\end{rmk}

We note that most classical properties of a well-pointed topos make use of projectivity and indecomposability of $1$, and many of these are still true intuitionistically as long as the category is \emph{constructively} well-pointed.

\begin{lem}\label{thm:pin-props}
  Let \bS\ be a constructively well-pointed Heyting category.  Then:
  \killspacingtrue
  \begin{enumerate}
  \item A morphism $p\colon Y\to X$ is regular epic if and only if every map $1\to X$ factors through it.\label{item:pp1}
  \item Likewise, $f\colon Y\to X$ is monic if and only if every map $1\to X$ factors through it in at most one way, and an isomorphism if and only if every map $1\to X$ factors through it uniquely.\label{item:pp1a}
  \item Given two subobjects $A\mono X$ and $B\mono X$, we have $A\le B$ if and only if every map $1\to X$ which factors through $A$ also factors through $B$.\label{item:pp2}
  \item $X$ is initial if and only if there does not exist any morphism $1\to X$.\label{item:pp3}
  \end{enumerate}
  \killspacingfalse
\end{lem}
\begin{proof}
  The ``only if'' part of~\ref{item:pp1} follows because $1$ is projective.
  For the ``if'' part, note that a map $p\colon Y\to X$ in a regular category is regular epic iff its image is all of $X$, while if every $1\to X$ factors through $p$ then it must factor through $\mathrm{im}(p)$ as well; hence $\mathrm{im}(p)$ is all of $X$ since $1$ is a strong generator.

  The ``only if'' parts of~\ref{item:pp1a} are obvious.
  If $f\colon Y\to X$ is injective on global elements, then the canonical monomorphism $Y\to Y\times_X Y$ is bijective on global elements, and hence an isomorphism; thus $Y$ is monic.
  And if $f$ is bijective on global elements, then by this and by~\ref{item:pp1} it must be both monic and regular epic, hence an isomorphism.

  The ``only if'' part of~\ref{item:pp2} is also obvious.
  If every $1\to X$ which factors through $A$ also factors through $B$, then in the pullback
  \[\vcenter{\xymatrix{A\cap B \pullbackcorner \ar[r]\ar[d] & B \ar[d]\\
      A\ar[r] & X}}
  \]
  the map $A\cap B\to A$ must be an isomorphism, since $1$ is a strong generator; hence its inverse provides a factorization of $A$ through $B$.

  Finally, the ``only if'' part of~\ref{item:pp3} is nonemptiness, while if $X$ has no global elements, the map $0\to X$ is bijective on global elements, hence an isomorphism.
\end{proof}

In particular, it follows that for constructively well-pointed categories, Booleanness coincides with ``two-valued-ness'':

\begin{lem}\label{thm:wp-bool-tv}
  Let \bS\ be a constructively well-pointed Heyting category; the following are equivalent.
  \begin{enumerate}
  \item $\bS$ is Boolean.\label{item:btv-bool}
  \item Every subterminal object is either terminal or initial.\label{item:btv-tv}
  \end{enumerate}
\end{lem}
\begin{proof}
  Assuming~\ref{item:btv-bool}, if $A$ is subterminal, then for its Heyting complement $\neg A$ we have $A\cup \neg A\cong 1$.
  By indecomposability of $1$, it follows that either $A$ is terminal or $\neg A$ is terminal, and the latter implies that $A$ is initial.

  Conversely, assuming~\ref{item:btv-tv}, let $A\mono X$ be a monomorphism and $\neg A \mono X$ its Heyting complement.
  Since $1$ is a strong generator, to show $A\cup\neg A \cong X$ it suffices to show every $x:1\to X$ factors through $A$ or $\neg A$.
  But given $x$, the pullback $x^*(A)$ is subterminal, hence either terminal or initial.
  In the former case, $x$ factors through $A$; while in the latter case it factors through $\neg A$.
\end{proof}

\begin{rmk}
  In terms of the internal logic, indecomposability of $1$ corresponds to the \emph{disjunction property} (if $\ph\vee\psi$ is true, then either $\ph$ is true or $\psi$ is true), while projectivity of $1$ corresponds to the \emph{existence property} (if $\exists x.\ph(x)$ is true, then there is a particular $a$ such that $\ph(a)$ is true).
  We might also call nonemptiness of $1$ the \emph{negation property} ($\bot$ is not true).
\end{rmk}


By a \emph{structural set theory}, we will informally mean an extension of the axiomatic theory describing a constructively well-pointed Heyting pretopos.
Not many structural set theories have been given particular names, but here are a few.
\begin{blist}
\item Lawvere's \textbf{Elementary Theory of the Category of Sets (ETCS)}, defined in~\cite{lawvere:etcs}, is the theory of a well-pointed topos with a \nno\ satisfying full classical logic and the axiom of choice.
\item Palmgren's \textbf{Constructive ETCS (CETCS)}, defined in~\cite{palmgren:cetcs}, is the theory of a constructively well-pointed $\Pi$-pretopos with a \nno\ and enough projectives.
\item We will refer to the theory of a constructively well-pointed topos with a \nno\ as \textbf{Intuitionistic ETCS (IETCS)}.
\end{blist}

\section{Structural set theory as set theory}
\label{sec:sst-as-setth}

So far, we have expressed the axioms of structural set theory in a form familiar to category theorists.
However, they can equivalently be formulated in a much more ``set-theoretic''  way, as was done in~\cite{lawvere:etcs}.
Besides its intrinsic interest, this reformulation facilitates the comparison with material set theories.

First of all, we formulate a structural counterpart of the notion of \ddo-formula.
For simplicity, we henceforth extend our language with a particular object-term ``$1$'', and modify the terminal-object axiom to assert simply that \qq{$1$ is terminal}.
Now we define:
\begin{blist}
\item An arrow-variable $x\colon 1\to X$ whose domain is $1$ is a \textbf{\ddo-variable}.
\item A context containing only \ddo-variables is a \textbf{\ddo-context}.
\item An equality $(f=g)$ of arrow-terms $f,g\colon 1\to X$ with source $1$ is a \textbf{\ddo-atomic formula}.
\item A quantifier over a \ddo-variable is a \textbf{\ddo-quantifier}.
\item A formula whose only atomic subformulas are \ddo-atomic and whose only quantifiers are \ddo-quantifiers is a \textbf{\ddo-formula}.
\end{blist}
Note that a \ddo-formula may contain non-\ddo\ \emph{free} variables.
We can now reformulate the basic structure of a Heyting category as a separation axiom.

\begin{prop}\label{thm:ddo-sep}
  Let \bS\ be a category with finite limits in which $1$ is a strong generator.  Then the following are equivalent.
  \begin{enumerate}
  \item \bS\ is a constructively well-pointed Heyting category.\label{item:ddosep-phc}
  \item \bS\ satisfies the schema of \emph{\ddo-separation}: for any \ddo-formula $\ph(x)$ with free \ddo-variable $x\colon 1\to X$, \qq{there exists a subobject $S\mono X$ such that any map $x\colon 1\to X$ factors through $S$ if and only if $\ph(x)$ holds}.\label{item:ddosep-sep}
  \end{enumerate}
\end{prop}
\begin{proof}
  First we assume~\ref{item:ddosep-phc} and prove~\ref{item:ddosep-sep}.
  As usual, in~\ref{item:ddosep-sep} we are allowing \ph\ to exist in an arbitrary context \Gm, and the variables occurring in \Gm\ are not restricted to be \ddo.
  To make the induction go through, we prove a more general statement.
  In addition to an arbitrary context \Gm, we suppose \Th\ to be a \ddo-extension of \Gm\ (that is, \Th\ consists of \ddo-variables and $(\Gm,\Th)$ is a context).

  Let $\rho\colon \Gm \to \bS$ be a variable assignment, and let $\rho(\Th)$ denote the product in \bS\ of the images under \rho\ of the codomains of all the \ddo-variables in \Th.
  For example, if $\Th = (x\colon 1\to X, y\colon 1\to Y)$, then $\rho(\Th) = \rho(X)\times \rho(Y)$.
  (If \bS\ does not have assigned products, we simply chose some such product.)
  Let $\ph(\Th)$ be a formula in context \Gm; we will prove, by structural induction on \ph, that there exists a subobject $S\mono \rho(\Th)$ such that $t\colon 1\to \rho(\Th)$ factors through $S$ if and only if $\ph(t)$ holds.
  This implies~\ref{item:ddosep-sep} by taking \Th\ to be a singleton.
  
  We begin by applying the pullback functor $\rho(\Th)^*\colon \bS \to \bS/\rho (\Th)$ to $\rho$, obtaining a new variable assignment $\rho_\Th\colon \Gm \to \bS/\rho(\Th)$.
  Now for any variable $x\colon 1\to X$ in \Th, the corresponding projection $\rho(\Th) \to \rho(X)$ induces a map
  \[\pi_x\colon 1_{\rho (\Th)}\to \rho(\Th)^*\rho(X) = \rho_\Th(X)\]
  in $\bS/\rho (\Th)$.
  We define $\rho_\Th' \colon (\Gm,\Th) \to \bS/\rho (\Th)$ to extend $\rho_\Th$ by $\rho_\Th'(x) = \pi_x$ for each such $x$.
  
  The variable assignment $\rho_\Th'$ allows us to interpret each object- and arrow-term in context $(\Gm,\Th)$ in $\bS/\rho (\Th)$.
  In particular, each \ddo-term is interpreted by a morphism with domain $\rho (\Th)$.
  If $f$ and $g$ are two such \ddo-terms, we let $\mm{f=g} \mono \rho(\Th)$ denote the equalizer of $\rho_\Th'(f)$ and $\rho_\Th'(g)$.
  It is clear that $t\colon 1\to \rho (\Th)$ factors through $\mm{f=g}$ if and only if $(f=g)(t)$ holds.

  Now we can build up $\mm{\ph}$ for all other formulas $\ph$ in the usual inductive way using the Heyting category structure of \bS, with intersections and unions interpreting $\meet$ and $\join$, the Heyting implication for $\imp$, images for $\im$, and dual images for $\coim$.
  It is easy to verify, using all parts of the definition of constructive well-pointedness, that at each step the desired property is maintained.

  Conversely, assume \ddo-separation~\ref{item:ddosep-sep} and that $1$ is a strong generator.
  The argument of \autoref{thm:pin-props}\ref{item:pp1a} still applies to show that a morphism is monic iff it is injective on global elements.
  Let \cM\ denote the class of monos and \cE\ the class of morphisms that are surjective on global elements.
  Because $1$ is a strong generator, any morphism in \cE\ is extremal epic (i.e.\ factors through no proper subobject of its codomain).
  Moreover, \cE\ is evidently stable under pullback.

  For any map $f\colon Y\to X$, apply \ddo-separation to the formula $\exists y\in Y. (f(y)=x)$ to obtain a monic $m\colon S\mono X$.
  The pullback of $m$ along $f$ is a monic which is bijective on global elements, hence an isomorphism; thus $f$ factors through $m$, and the factorization $e\colon Y\to S$ is in \cE\ by construction.
  Therefore, $(\cE,\cM)$ is a pullback-stable factorization system, from which it follows that \bS\ is a regular category and that \cE\ is exactly the class of regular epics---and hence $1$ is projective.

  Now given monos $m\colon A\mono X$ and $n\colon B\mono X$, apply \ddo-separation to the formula $\exists a\in A.(m(a)=x) \join \exists b\in B.(n(b)=x)$ to obtain a monic $S\mono X$.
  \autoref{thm:pin-props}\ref{item:pp2} still applies to show that $A\subseteq S$ and $B\subseteq S$ and that if $A\subseteq T$ and $B\subseteq T$ then $S\subseteq T$; thus $S=A\cup B$.
  The defining property of $S$ (it contains precisely those $1\to X$ factoring through either $A$ or $B$) is evidently stable under pullback.
  Similarly, from the formula $\ph(x) = \bot$ we obtain a pullback-stable bottom element $0\mono X$, so \bS\ is a coherent category.
  Moreover, by the construction of unions and empty subobjects, it follows that $1$ is indecomposable and nonempty.
  Finally, dual images can similarly be constructed by applying \ddo-separation to a formula with a universal \ddo-quantifier.
\end{proof}


Further categorical properties can also be reformulated in set-theoretic terms.

\begin{prop}\label{thm:boolean}
  A constructively well-pointed Heyting category is Boolean if and only if it satisfies \emph{\ddo-classical logic}: $\ph\vee\neg\ph$ for all \ddo-formulas $\ph$.
\end{prop}
\begin{proof}
  For a subobject $S\mono A$, let $\neg S \mono A$ be the Heyting negation.
  By well-pointedness, $x\colon 1\to A$ factors through $\neg S$ if and only if it does not factor through $S$.
  Since \qq{$x$ factors through $S$} is \ddo, giving \ddo-classical logic every such $x$ factors through $S$ or $\neg S$; so by well-pointedness, $S\cup \neg S = A$ and $S$ is complemented.

  Conversely, if the category is Boolean, then any \ddo-formula \ph\ has a classifying subobject $\mm{\ph}$, and negation of formulas corresponds to the Heyting negation.
  Thus, if $\mm{\ph}$ is complemented, $\ph \vee \neg\ph$ must be valid.
\end{proof}

\begin{prop}\label{thm:wp-powerobj}
  A constructively well-pointed Heyting category is a topos if and only if it satisfies the following axiom:
  \begin{blist}
  \item \qq{for any $X$ there is an object $P X$ and a relation $[\ordin_X] \mono X\times P X$ such that for any $S\mono X$, there exists a unique $s\colon 1\to  P X$ such that for all $x\colon 1\to X$, we have that $x$ factors through $S$ if and only if $\pair(x,s)$ factors through $[\ordin_X]$}.
  \end{blist}
\end{prop}
\begin{proof}
  The necessity of this axiom follows from the universal property of power objects.
  Conversely, assuming it, let $R\mono X\times Y$ be any subobject.
  Using \ddo-separation, let $F\mono Y\times P X$ contain exactly those pairs $\pair(y,s)$ such that for all $x\colon 1\to  X$, we have that  $\pair(x,y)$ factors through $R$ if and only if $\pair(x,s)$ factors through $[\ordin_X]$.
  By the given axiom, for any $y\colon 1 \to Y$ there is a unique $s\colon 1\to P X$ such that $x$ factors through $y^*R$ if and only if $\pair(x,s)$ factors through $[\ordin_X]$, and this if and only if $\pair(y,s)$ factors through $F$.
  Therefore, since $1$ is a strong generator, the projection $F\to Y$ is invertible, and so $F$ defines a map $f\colon Y\to P X$.
  By \autoref{thm:pin-props}\ref{item:pp2}, we have $f^*([\ordin_X]) \cong R$, and uniqueness is likewise easy.
\end{proof}

\begin{prop}\label{thm:wp-lcc}
  A constructively well-pointed Heyting category is cartesian closed if and only if it satisfies the following axiom:
  \begin{blist}
  \item \qq{For any $X$ and $Y$, there is an object $Y^X$ and an arrow $\mathrm{ev}\colon Y^X \times X\to Y$ such that for any $f\colon X\to Y$, there is a unique $z\colon 1\to Y^X$ such that for all $x\colon 1\to X$ we have $\mathrm{ev}\circ \pair(z,x) = f \circ x$}.
  \end{blist}
\end{prop}
\begin{proof}
  Analogous.
\end{proof}

We will see in \autoref{thm:cc-lcc} that in the presence of a structural replacement axiom, the axiom in \autoref{thm:wp-lcc} also implies \emph{local} cartesian closure, just as with limited \ddo-replacement and exponentiation in material set theory.

\begin{rmk}
  The expressive power of \ddo-formulas depends on the other categorical structure that is present.
  For instance, in a cartesian closed category, morphisms $f\colon X\to Y$ are equivalent to morphisms $1 \to Y^X$, so \ddo-separation implies separation for formulas including quantifiers over arbitrary arrow-variables.
  Similarly, in a topos we can effectively quantify over subobjects of any $A$ (that is, objects $S$ equipped with a monomorphism $S\mono A$) by quantifying over morphisms $1\to P A$; the mono $S\mono A$ is determined only up to isomorphism by its classifying map $1\to P A$, but by isomorphism-invariance (\autoref{thm:isoinvar-truth}) this is invisible to our logic.
\end{rmk}

\begin{cnv}\label{cnv:elements}
  From now on we will use the following more ``set-theoretic'' terminology and notation for structural set theories.
  We write \Set\ for the category in question, and call its objects \emph{sets} and its arrows \emph{functions}.
  We speak of morphisms $1\to X$ as \emph{elements} of $X$, and we write $x\cin X$ to mean $x\maps 1\to X$.
  If $f\colon X\to Y$ is a function and $x\cin X$, we write $f(x)\cin Y$ for $f\circ x\colon 1\to Y$.
  Similarly, we speak of monomorphisms $S\mono X$ as \emph{subsets} of $X$ and write $S\csub X$.

  The notations $\cin$ and $\csub$ are a hybrid of the familiar set-theoretic $\in$ and $\subseteq$ with the type-theoretic $x\colon A$.
  Unlike the former, but like the latter,
  they are not propositions of the theory that can be true or false, but rather typing judgments which indicate what sort of thing their left-hand argument is.
  By contrast, given $x\cin X$ and $S\csub X$, we write $x\in S$ for the statement \qq{$1\too[x] X$ factors through $S\mono X$}, which \emph{is} a proposition of the theory.
  Similarly, for $S,T\csub X$, we write $S\subseteq T$ for the statement \qq{$S\mono X$ factors through $T\mono X$}.
\end{cnv}

\section{Structural separation, fullness, and induction}
\label{sec:more-sst}

There are several axioms from material set theory for which we have not yet considered structural versions.
In this section, we consider full separation, fullness, and induction; in the next we consider collection and replacement.

The structural axiom of full separation simply generalizes \autoref{thm:ddo-sep}\ref{item:ddosep-sep} to unbounded quantifiers.
\begin{blist}
\item \emph{Separation:} For any formula $\ph(x)$, \qq{for any set $X$, there exists $S\csub X$ such that for any $x\cin X$, we have $x\in S$ if and only if $\ph(x)$}.
\end{blist}

\begin{lem}\label{thm:set-sep}
  If \bV\ satisfies the core axioms together with full separation, then $\bbSet(\bV)$ satisfies separation.
\end{lem}
\begin{proof}
  As observed in \autoref{rmk:translation}, any formula \ph\ in $\bbSet(\bV)$ in the language of categories may be translated into a formula \phhat\ in \bV\ in the language of material set theory.
  If \bV\ satisfies material separation, then $S = \setof{x\in X | \phhat(x)}$ has the desired property for structural separation.
\end{proof}

The structural axiom of fullness is also a direct translation of the material one.
\begin{blist}
\item \emph{Fullness:} \qq{for any sets $X,Y$ there exists a relation $R\csub M\times X\times Y$ such that $R\epi M\times X$ is regular epic, and for any relation $S\csub X\times Y$ such that $S\epi X$ is regular epic, there exists an $s\cin M$ such that $(s,1)^*R \subseteq S$}.
\end{blist}

\begin{lem}\label{thm:set-fullness}
  If \bV\ satisfies the core axioms together with fullness, then $\bbSet(\bV)$ satisfies structural fullness.
\end{lem}
\begin{proof}
  Just like the proofs for exponentials and power sets in \autoref{thm:iz-topos}.
\end{proof}

\begin{rmk}
  We have stated the axiom of fullness in a form referring explicitly to global elements, analogous to the axioms appearing in Propositions \ref{thm:wp-powerobj} and \ref{thm:wp-lcc}.
  Thus, unlike the properties of being a topos or being locally cartesian closed, this axiom is only suitable in a well-pointed category.
  However, by interpreting the above axiom in the stack semantics of~\cite{shulman:stacksem}, we can obtain a fullness axiom that makes sense more generally.
  (In fact, it is exactly the categorical version of fullness from~\cite{vdbdb:nonwellfounded,vdbm:pred-i-exact}, but lacking their `smallness' conditions).
\end{rmk}

The axiom of induction, of course, only makes sense in the presence of the axiom of infinity.
The structural axiom of infinity (existence of an \nno) asserts that functions can be constructed by recursion, which implies Peano's induction axiom in the sense that any subset $S\csub N$ which contains $0\cin N$ and is closed under $s\colon N\to N$ must be all of $N$.
(The proof of \autoref{thm:iz-topos}\ref{item:iz-topos-nno} essentially shows that the converse holds in any \Pi-pretopos.)
It follows from \ddo-separation that \ddo-formulas can be proven by induction; the axiom of full induction extends this to arbitrary formulas.
\begin{blist}
\item \emph{Induction:} For any formula $\ph(x)$ with free variable $x\cin N$, where $N$ is an \nno, \qq{if $\ph(0)$ and $\forall n\cin N.(\ph(x)\imp \ph(s x))$, then $\ph(x)$ for all $x\cin N$}.
\end{blist}
Just as in material set theory, infinity and full separation together imply full induction.  We also have:

\begin{prop}\label{thm:set-ind}
  If \bV\ satisfies the core axioms of material set theory, and also the axioms of infinity, exponentials, and induction, then $\bbSet(\bV)$ has an \nno\ and satisfies induction.
\end{prop}
\begin{proof}
  As in the proof of \autoref{thm:set-sep}, any formula $\ph(x)$ in $\bbSet(\bV)$ with $x\cin N$ (where $N = \omega$) can be rewritten as a formula in \bV, to which the material axiom of induction can be applied.
\end{proof}

There is no particularly natural structural ``axiom of foundation,'' although in \S\ref{sec:conclusion} we will mention a somewhat related property.
We can, however, formulate structural axioms which are closely related to the material axiom of set-induction.
\begin{blist}
\item \emph{Well-founded induction:} For any formula $\ph(x)$ with free variable $x\cin A$, \qq{if $A$ is well-founded under the relation $\prec$, and moreover if $\ph(y)$ for all $y\prec x$ implies $\ph(x)$, then in fact $\ph(x)$ for all $x\cin A$}.
\item \emph{Extensional well-founded induction:} The same, but requiring $A$ to also be extensional.
\end{blist}
These axioms are implied by separation, since then we can form $\setof{ x\cin A | \ph(x) }$ and apply the definition of well-foundedness.
We can also say:

\begin{prop}
  If \bV\ satisfies the core axioms of material set theory, and also the axioms of set-induction and Mostowski's principle, then $\bbSet(\bV)$ satisfies extensional well-founded induction.
  If \bV\ additionally has power sets, then $\bbSet(\bV)$ satisfies well-founded induction.
\end{prop}
\begin{proof}
  By Mostowski's principle, any extensional well-founded relation is isomorphic to a transitive set.
  Thus, set-induction performed over the resulting transitive set can be used for inductive proofs over the original well-founded relation.
  Finally, with power sets, any well-founded relation can be collapsed to a well-founded extensional one (an ``extensional quotient'' in the sense of \S\ref{sec:constr-mat}).
\end{proof}

\section{Structural collection and replacement}
\label{sec:strong-ax}

We now turn to structural versions of the collection and replacement axioms.
Various such axioms have been proposed in the context of ETCS (see~\cite{lawvere:etcs,cole:cat-sets,osius:cat-setth,mclarty:catstruct}), but none of these seem to be quite appropriate in an intuitionistic or predicative theory.
Hence our axioms must be different from all previous proposals (though they are similar to the replacement axiom of~\cite{mclarty:catstruct}).

The intuition behind structural collection is that since the elements of a set in a structural theory are not themselves sets, instead of ``collecting'' sets as \emph{elements} of another set we must collect them as a family indexed \emph{over} another set.
Also, since the language of category theory is two-sorted, it is unsurprising that we have to assert collection for objects and morphisms separately.
In fact, we find it conceptually helpful to formulate \emph{three} axioms of collection, although the third one is automatically satisfied.

Remember that all formulas exist in an unstated ambient context \Gm.
(In particular, the object-variable $U$ in the first two axioms below must belong to \Gm.)
\begin{blist}
\item \emph{Collection of sets:} For any formula $\ph(u,X)$ with specified free variables $u\cin U$ and $X$,
  \qq{if for every $u$ there exists an $X$ with $\ph(u,X)$, then there exists a regular epi $V\xepi{p} U$ and an $A$ in $\Set/V$ such that for every $v \cin V$ we have $\ph(pv,v^*A)$}.
\item \emph{Collection of functions:} For any formula $\ph(u,X,Y,f)$, with specified free variables $u\cin U$, $X$, $Y$, and $f\colon X\to Y$,
  \qq{for any $A,B$ in $\Set/U$, if for all $u\cin U$ there exists $u^*A\too[f] u^*B$ with $\ph(u,u^*A,u^*B,f)$, then there exists a regular epi $V\xepi{p} U$ and a function $p^*A\too[g] p^*B$ in $\Set/V$ such that for all $v\cin V$, we have $\ph(pv,(pv)^*A, (pv)^*B, v^*g)$}.
\item \emph{Collection of equalities:}
  \qq{For any set $U$, any $A,B$ in $\Set/U$, and any functions $f,g\colon A\to B$ in $\Set/U$, if $u^*f = u^*g$ for every $u\cin U$, then there is a regular epi $V\xepi{p} U$ such that $p^*f = p^*g$.}
\end{blist}
We will abbreviate a formula $\ph(u,X,Y,f)$ as in the axiom of collection of functions by $\ph(u,f)$, with the object-variables $X$ and $Y$ implicitly present as the domain and codomain of $f$.

We will say that \Set\ ``satisfies collection'' if it satisfies all three of the above axioms.
However, one, and sometimes two, of these are redundant.

\begin{prop}\label{thm:collmor}
    Let \Set\ be a constructively well-pointed Heyting category.
    \killspacingtrue
    \begin{enumerate}
    \item \Set\ satisfies collection of equalities.\label{thm:colleq}
    \item If \Set\ satisfies full separation and is a \Pi-pretopos, then it satisfies collection of functions.\label{item:collmor2}
    \end{enumerate}
    \killspacingfalse
\end{prop}
\begin{proof}
  To show~\ref{thm:colleq},
  let $E \xto{e} A$ be the equalizer of $f$ and $g$ and let $S$ be the dual image $\coim_a E$, where $A\xto{a} U$ is the structure map of $A$ as an object of $\Set/U$.
  The assumption that $u^*f = u^*g$ for every $u\cin U$ implies that every such $u$ is contained in $S$.
  Since $1$ is a strong generator, this implies $S\cong U$, and therefore $E\cong A$ and so $f=g$; thus we can take $V=U$ and $p=1_U$.

  To show~\ref{item:collmor2}, let $C= (B\to U)^{(A\to U)}$ be the exponential in $\Set/U$, with projection $C \too[g] U$.
  Then each $c\cin C$ corresponds to a map $f_c\maps (gc)^*A \to (gc)^*B$.
  Using the axiom of separation, find a subobject $V\xmono{m} C$ such that $c\cin C$ is contained in $V$ iff $\ph(f_c)$.
  By assumption, every $u\cin U$ lifts to some $c\cin V$, so since \Set\ is well-pointed, the projection $V\xepi{gm} U$ is regular epi.
  Finally, there is an evident map $h\maps (gm)^*A\to (gm)^*B$ such that $\ph(gmc,c^*h)$ for all $c\cin V$.
\end{proof}

The appropriate structural formulation of the axiom of \emph{replacement} is a bit more subtle.
The idea is that if we modify the hypotheses of collection by asserting \emph{unique} existence, then passage to a cover $V\epi U$ should be unnecessary.
As with collection, we may expect three versions for sets, functions and equalities.
The second two of these are easy to state, and follow from collection.

\begin{prop}\label{thm:replacement}
  Let \Set\ be a constructively well-pointed Heyting category.
  \killspacingtrue
  \begin{enumerate}
  \item \Set\ always satisfies \emph{replacement of equalities}: \qq{for any set $U$, any $A,B$ in $\Set/U$, and any functions $f,g\colon A\to B$ in $\Set/U$, if $u^*f = u^*g$ for every $u\cin U$, then $f=g$}.\label{item:repl-eq}
  \item If \Set\ satisfies collection of functions, then it satisfies \emph{replacement of functions}: For any formula $\ph(u,f)$ as in collection of functions,
    \qq{for any $A,B$ in $\Set/U$, if for every $u\cin U$ there exists a unique $u^*A \too[f] u^*B$ such that $\ph(u, f)$, then there exists $A\too[g] B$ in $\Set/U$ such that for each $u\cin U$ we have $\ph(u, u^*f)$}.\label{item:repl-mor}
  \end{enumerate}
  \killspacingfalse
\end{prop}
\begin{proof}
  The proof of \autoref{thm:collmor}\ref{thm:colleq} already shows~\ref{item:repl-eq}.
  For~\ref{item:repl-mor}, collection of functions gives us a regular epimorphism $V\xepi{p} U$ and a function $h\colon p^*A \to p^*B$ in $\Set/V$ such that $\ph(p v, v^*h)$ for any $v\cin V$.
  If $(r,s)\colon V\times_U V\toto V$ is the kernel pair of $p$, then for every $z = (v_1,v_2)\cin V\times_U V$ we have $p v_1 = p v_2 = u$, say, and thus $\ph(u, v_1^*h)$ and $\ph(u, v_2^*h)$.
  By uniqueness, $v_1^*h = v_2^*h$, and so
  \[z^*r^*h = v_1^* h = v_2^* h = z^* s^* h.
  \]
  By replacement of equalities, $r^*h = s^*h$.
  But the self-indexing of \Set\ is a prestack, so $h$ descends to $g\colon A\to B$ with $p^*g = h$.
  Since $p$ is regular epic, for any $u\cin U$ there is a $v\cin V$ with $p v = u$; thus $u^* h = v^* p^* h = v^* g$, whence $\ph(u,u^*h)$.
\end{proof}

These two replacement axioms imply, in particular, that \emph{universal properties are reflected by global elements}.
Rather than give a precise statement of this, we present a paradigmatic example.

\begin{prop}\label{thm:cc-lcc-mor}
  Let \Set\ be a constructively well-pointed Heyting category satisfying replacement of functions, and suppose we have $A,B,E$ in $\Set/X$ and a morphism $v\colon E\times_X A\to B$ such that for all $x\cin X$, $x^*v$ exhibits $x^*E$ as an exponential $(x^*B)^{(x^*A)}$ in \Set.
  Then $v$ exhibits $E$ as an exponential $B^A$ in $\Set/X$.
\end{prop}
\begin{proof}
  The assumption means that for any object $C$ and any morphism $f\colon C\times x^*A \to x^*B$, there exists a unique $g\colon C\to x^*E$ such that $f = x^*v \circ (g\times 1)$.
  In particular, given any $D$ in $\Set/X$ and morphism $f\colon D\times_X A\to B$, for every $x\cin X$ there is a unique $g_x\colon x^*D\to x^*E$ such that $x^*f = x^*v \circ (g_x\times 1)$.
  By replacement of functions, we have $g\colon D\to E$ such that $x^*f = x^*v \circ (x^*g \times 1)$ for each $x\cin X$, or equivalently $x^* f = x^*(v\circ (g\times 1))$.
  By replacement of equalities, $f=v\circ (g\times 1)$.
  Moreover, if $f = v\circ (h\times 1)$ for some $h\colon D\to E$, then for each $x\cin X$ we have $x^*f = x^*v \circ (x^*h \times 1)$.
  By the universal property of the exponential $x^*E$, we have $x^*h = g_x = x^*g$, and hence $h=g$ by replacement of equalities.
\end{proof}

Note that all such universal properties are automatically pullback-stable.

We would also like an axiom of ``replacement of sets,'' which would ensure that the \emph{existence} of an object satisfying a universal property is also reflected by global elements.
Since no structural theory can determine an object of a category more uniquely than up to unique isomorphism, one natural such statement would be:
\begin{blist}
\item \emph{Replacement of sets:}
  \qq{if for every $u\cin U$ there is a set $A$ with $\ph(u,A)$ which is unique up to unique isomorphism, then there is a $B$ in $\Set/U$ such that $\ph(u,u^*B)$ for all $u\cin U$}.
\end{blist}
If \Set\ is a pretopos (and in particular exact), then replacement of sets follows from collection by a proof analogous to that of \autoref{thm:replacement}, using the fact that in this case the self-indexing is a stack for the regular topology.
Moreover, just as in material set theory, we have:

\begin{prop}\label{thm:collsep}
  If \bS\ is a well-pointed Heyting category which satisfies full classical logic and replacement of sets, then it also satisfies full separation.
\end{prop}
\begin{proof}
  For any formula $\ph(x)$ with free variable $x\cin X$, let $\psi(x,A)$ be \qq{$A$ is terminal and $\ph(x)$, or $A$ is initial and $\neg\ph(x)$}.
  Since $\forall x.(\ph(x)\vee\neg\ph(x))$ holds by classical logic, and initial and terminal objects are unique up to unique isomorphism, by replacement of sets we have a $B$ in $\Set/X$ such that for any $x\cin X$, if $\ph(x)$ then $x^*B$ is terminal, while if $\neg\ph(x)$ then $x^*B$ is initial.
  It follows that $B\to X$ is monic, and is the subobject classifying $\ph$ required by the separation axiom.
\end{proof}

However, for most other purposes replacement of sets is basically useless, since hardly ever is a set determined up to \emph{absolutely} unique isomorphism, only an isomorphism which is compatible with some additional functions whose existence is being asserted at the same time.
For example, a cartesian product $A\times B$ is only determined up to an isomorphism which is unique \emph{such that} it respects the projections $A\times B\to A$ and $A\times B\to B$ with which the product comes equipped.

In order to state a useful version of replacement, we recall the notion of \emph{context extension} from \autoref{notn:struct-context}: if $(\Gm,\Th)$ is a context, then we can write $\ph(\Theta)$ in context \Gm\ if \ph\ is a formula whose free variables all lie in $(\Gm,\Th)$.
Since all contexts are finite, we can write formulas like \qq{there exists \Th\dots} or \qq{for all \Th\dots} in context $\Gm$, abbreviating a finite list of ordinary quantifiers.
Similarly, if $(\Gm,\Xi,\Th)$ is a context, then for a given $\Xi$ in context \Gm\ we can write \qq{there exists \Th\ extending \Xi\dots}.
For instance, if $\Xi = (X,Y)$ and $\Th = (f\colon X \to Y)$, then \qq{there exists \Th\ extending $(A,B)$\dots} means \qq{there exists $f\colon A\to B$\dots}.

We will write $\Th'$ for a context obtained from \Th\ by replacing all of its variables by fresh ones (hereafter denoted with primes, $X\mapsto X'$), so that $(\Gm,\Th,\Th')$ is also a context.
We write $\Phi_\Th$ for a list of arrow-variables $\phi_X\colon X\to X'$, one for each object-variable $X$ in \Th; thus $(\Gm,\Th,\Th',\Phi_\Th)$ is also a context.
Finally, we will write
\begin{center}
\qq{$\Phi_\Th$ is an isomorphism $\Th\cong\Th'$}
\end{center}
for the formula $\ph(\Th,\Th',\Phi_\Th)$ in context \Gm\ which asserts that
\begin{enumerate}
\item for every object-variable $X$ in \Th, $\phi_X$ is an isomorphism; and
\item for every arrow-variable $f\colon X\to Y$ in \Th, we have $\phi_Y \circ f = f' \circ \phi_X$.
\end{enumerate}
If $X$ is in $\Gm$ rather than \Th, then $\phi_X$ denotes the identity $1_X$, and likewise for $Y$.

We can also talk about contexts and extensions of contexts in slice categories.
A particularly important example is that the context $(X,x\cin X)$ can be instantiated in $\Set/U$ by the pair $(U^*U,\Delta_U)$ consisting of the object $U^*U = U\times U$ and the morphism $\Delta_U\colon 1_U\to U^*U$.
Moreover, this pair is the \emph{universal} such instantiation over $U$, in that for any $u\colon 1\to U$ we have an isomorphism $u^*(U^*U,\Delta_U) \cong (U,u)$.
Similarly, for any $V\too[p] U$ we can consider $(V^*U,(1_V,p))$, which has the analogous property that $v^*(V^*U,(1_V,p)) \cong (U, p v)$ for any $v\colon 1\to V$.

\begin{prop}\label{thm:coll-ctxt}
  A constructively well-pointed Heyting category satisfying collection also satisfies \emph{collection of contexts:} for any formula $\ph(U,u\cin U,\Th)$ (in any context $\Gm$, as always),
  \begin{blist}
  \item \qq{for every $U$, if for every $u\cin U$ there exists $\Th$ extending $(U,u)$ such that $\ph(U,u,\Th)$, then there exists a regular epi $V\xepi{p} U$ and a $\Th'$ in $\Set/V$ extending $(V^*U,(1_V,p))$ such that for every $v \cin V$ we have $\ph(U,pv,v^*\Th')$}.
\end{blist}
\end{prop}
\begin{proof}
  Simply apply collection of sets and functions repeatedly, using the fact that regular epimorphisms compose.
\end{proof}

We can now improve ``replacement of sets'' to a more useful axiom.
\begin{blist}
\item \emph{Replacement of contexts:} for any formula $\ph(U,u\cin U,\Th)$,
  \qq{for any $U$, if for every $u\cin U$ there is an extension $\Th$ of $(U,u)$ such that $\ph(U,u,\Th)$, and moreover this extension $\Th$ is unique up to a unique isomorphism of contexts $\Phi_\Th$, then there is an extension $\Th'$ of $(U^*U, \Delta_U)$ in $\Set/U$ such that $\ph(U,u,u^*\Th')$ for all $u\cin U$}.
\end{blist}

\begin{prop}\label{thm:repl-ctxt}
  Any constructively well-pointed Heyting pretopos satisfying collection also satisfies replacement of contexts.
\end{prop}
\begin{proof}
  By \autoref{thm:coll-ctxt}, there is a regular epi $p\colon V\epi U$ and a $\Th'$ in $\Set/V$ such that for every $v \cin V$ we have $\ph(U,pv,v^*\Th')$.
  Let $(r,s)\colon V\times_U V \toto V$ be the kernel pair of $f$, and consider $r^*\Th'$ and $s^*\Th'$ in $\Set/V\times_U V$.
  The assumption implies that when pulled back along any $z\colon 1\to V\times_U V$, these two contexts become uniquely isomorphic.
  Thus, by replacement of functions and equalities applied some finite number of times, we actually have a unique isomorphism $r^*\Th'\cong s^*\Th'$ in $\Set/V\times_U V$.
  Uniqueness implies that this isomorphism satisfies the cocycle condition over $V\times_U V\times_U V$.
  Thus, since \Set\ is a stack for its regular topology (applied some finite number of times), the entire context $\Th'$ descends to some $\Th''$ in $\Set/U$.
  Since $p$ is surjective, the desired property for $\Th''$ follows.
\end{proof}

I do not know whether
replacement of contexts implies collection. 
But as in material set theory, replacement suffices in the classical world.

\begin{prop}\label{thm:class-repcoll}
  Over ETCS, the following are equivalent.
  \begin{enumerate}
  \item Replacement of contexts.
  \item Collection of sets.
  \item Full separation and collection (of both sets and functions).
  \end{enumerate}
\end{prop}
\begin{proof}
  By Propositions \ref{thm:collmor}\ref{item:collmor2}, \ref{thm:collsep}, and \ref{thm:repl-ctxt}, it suffices to prove that replacement of contexts implies collection of sets.
  Given a formula $\ph(u,X)$, we can construct a context \Th\ including both a set $X$ and a binary relation on it, and let $\psi(u,\Th)$ assert that $\ph(u,X)$ holds and that the given relation is a well-ordering, which has the smallest possible order-type among well-ordered sets $X$ such that $\ph(u,X)$.
  Since classically, any set admits a smallest well-ordering which is unique up to unique isomorphism, we can apply replacement of contexts to $\psi$ and thereby deduce collection of sets.
\end{proof}

As promised, replacement of contexts implies that \emph{the existence of objects satisfying a universal property is reflected by global elements}.
Continuing the example of \autoref{thm:cc-lcc-mor}, we have the following.

\begin{prop}\label{thm:cc-lcc}
  Let \Set\ be a constructively well-pointed Heyting pretopos satisfying replacement of contexts.
  If \Set\ is cartesian closed, then it is locally cartesian closed.
\end{prop}
\begin{proof}
  Let $A,B$ be in $\Set/X$; we want to construct the exponential $(B\to X)^{(A\to X)}$.
  For $x\cin X$, let $\ph(x,E,e)$ assert that $e\colon E\times x^*A \to x^*B$ exhibits $E$ as the exponential $(x^*B)^{x^*A}$.
  This formula can be phrased as $\psi(X,x,\Th)$ for some \Th, although we note that \Th\ involves not just $E$ and $e$ but also object-variables for $x^*A$, $x^*B$, and $E\times x^*A$, and arrow-variables giving the projections that exhibit these as two pullbacks and a cartesian product, respectively.

  Now, since \Set\ is cartesian closed, for every $x\cin X$ there exists $E$ and $e$ with $\ph(x,E,e)$, and such are unique up to a unique isomorphism which respects all the structure.
  Therefore, replacement of contexts supplies an instantiation of \Th\ extending $(X^*X, \Delta_X)$ in $\Set/X$.
  This consists essentially of an object $C$ and a morphism $c\colon C\times_X A \to B$.
  The conclusion of replacement of contexts implies that $C$ and $c$ satisfy the hypotheses of \autoref{thm:cc-lcc-mor}, so they must be an exponential $(B\to X)^{(A\to X)}$ as desired.
\end{proof}

We end this section by extending \autoref{thm:iz-topos} to the axiom of collection.

\begin{thm}\label{thm:iz-ext-seprepcoll}
  If \bV\ satisfies the core axioms of material set theory along with collection, then $\bbSet(\bV)$ satisfies structural collection.
\end{thm}
\begin{proof}
  If \bV\ satisfies material collection, then given the setup of collection of sets for some formula $\ph$, let $\psi(u,X)$ assert that $X$ is a Kuratowski ordered pair of the form $(u,X')$ with $\ph(u,X')$.
  By material collection, let $V$ be a set such that for any $u\in U$, there is an $X\in V$ with $\psi(u,X)$, and for any $X\in V$, there is a $u\in U$ with $\psi(u,X)$.
  Then every element of $V$ is an ordered pair whose first element is in $U$, so there is a projection $V\to U$, which by assumption is surjective.
  We can then form $B = \bigcup_{(u,X')\in V} X'$, and by \ddo-separation cut out
  \[ A = \setof{((u,X'),a) \in V\times B | a \in X'}.
  \]
  Then $A$ is an object of $\bbSet(\bV)/V$, and for any $v=(u,X')\in V$, $v^*A$ is isomorphic to $X'$.
  Thus, by isomorphism-invariance, we have $\ph(u,v^*A)$, as required.
  Collection of functions is analogous.
\end{proof}

As far as I can tell, however, the material axiom of replacement does not seem to imply the structural axiom of replacement of contexts (in the absence of classical logic).
This, I feel, is one of the reasons (though not the only one) that the material axiom of collection is necessary in practice.
For example, it is well-known that while some types of ``infinite iteration'' can be performed in a topos with a \nno, and consequently some sorts of ``free structures'' can be proven to exist (c.f.~\cite{jw:at-toposes}), this is not always the case.
However, infinite iteration (and more generally, transfinite iteration over a well-founded relation) is always possible in a well-pointed topos satisfying our structural axioms of induction and replacement.
In~\cite{shulman:stacksem} we will see how to extend this observation to the non-well-pointed case.

\section{Constructing material set theories}
\label{sec:constr-mat}

So far we have summarized the axioms of material set theory and structural set theory, and explained how most axioms of material set theory are reflected in structural properties of the resulting category of sets.
We now turn to the opposite construction: how to recover a material set theory from a structural one.
A ``set'' in material set theory, of course, contains much more information than a ``set'' in structural set theory, namely the membership relations between its elements, their elements, and so on.
This gives rise to the idea of modeling a ``material set'' by a \emph{graph} or \emph{tree} with nodes depicting sets and edges depicting membership.

This basic idea was used by~\cite{cole:cat-sets,mitchell:topoi-sets,osius:cat-setth} in the first equiconsistency proofs of ETCS with BZC, and nearly identical constructions have been used for relative consistency proofs by others such as~\cite{aczel:afa} and~\cite{mathias:str-maclane}.
Here we perform essentially the same construction, but with an attention to details that ensures everything works in an intuitionistic or predicative context.

For the rest of this section, we work in a constructively well-pointed \Pi-pretopos with an \nno, denoted \Set.
I do not know whether an analogous construction can be made to work under weaker assumptions than this.

\begin{defns}\ 
  \killspacingtrue
  \begin{enumerate}
  \item A \textbf{graph} is a set $X$, whose elements are called \textbf{nodes}, equipped with a binary relation $\prec$.
  \item If $x\prec y$ we say that $x$ is a \textbf{child} of $y$.
  \item A \textbf{pointed} graph is one equipped with a distinguished node $\star$ called the \textbf{root}.
  \item A pointed graph is \textbf{accessible} if for every node $x$ there exists a path $x = x_n \prec \cdots \prec x_0 = \star$ to the root.
  \item For any node $x$ of a graph $X$, we write $X/x$ for the full subgraph of $X$ consisting of those $y$ admitting some path to $x$.
    It is, of course, pointed by $x$, and accessible.
  \end{enumerate}
  \killspacingfalse
  ``Accessible pointed graph'' is abbreviated \apg.
\end{defns}

\begin{rmk}
  The hypothesis on \Set\ is necessary to formalize ``accessibility'' and to define $X/x$.
  Specifically, a pointed graph $X_{\prec} \mono X\times X$ with root $\star\cin X$ is accessible if for every $x\cin X$, there exists a nonzero finite cardinal $[n]$ and a map $[n] \to X_{\prec}$ realizing a path from $x$ to $\star$.
  (Recall that a finite cardinal is the pullback along a map $1\to N$ of the second projection $N_< \to N$ of the strict order relation on the \nno.)
  The definition of $X/x$ is similar, using the fact that \ddo-separation can be applied to functions in a \Pi-pretopos.
\end{rmk}

We are using the terminology of~\cite{aczel:afa}.  The idea is that an
arbitrary graph represents a collection of material-sets with $\prec$
representing the membership relation between them, a pointed graph
represents a \emph{particular} set (the root) together with all the
data required to describe its hereditary membership relation, and an
\apg\ does this without any superfluous data (all the nodes bear some
relation to the root).  Thus, an arbitrary \apg\ can be considered a
picture of a possibly non-well-founded set: the root represents the
set itself, its children represent the elements of the set, and so on.

We will henceforth restrict attention to graphs for which $\prec$ is
well-founded, thereby ensuring that the models of material set theory
we construct satisfy the axiom of foundation.  It is also possible, by
changing this requirement, to construct models satisfying the various
axioms of anti-foundation (see~\cite{aczel:afa}), but we will not do
that here.  Recall the definition:

\begin{defn}\label{defn:well-founded}
  A subset $S$ of a graph $X$ is \textbf{inductive} if for any node
  $x\cin X$, if all children of $x$ are in $S$, then $x$ is also in
  $S$.  A graph $X$ is \textbf{well-founded} if any inductive subset
  of $X$ is equal to all of $X$.
\end{defn}

In the presence of classical logic, well-foundedness is equivalent to
saying that every inhabited subset of $X$ has a $\prec$-minimal element,
but in intuitionistic logic that version is too strong:

\begin{prop}\label{thm:clwf-lem}
  If there exists a graph $X$ containing nodes $u,v$ with $u\prec v$, and such that every inhabited subset $S\csub X$ has a $\prec$-minimal element (i.e.\ there exists an $x\in S$ such that $y\nprec x$ for all $y\in S$), then excluded middle holds (i.e.\ $\Set$ is Boolean).
\end{prop}
\begin{proof}
  I learned this proof from Zhen Lin Low and Daniil Frumin.
  By \autoref{thm:wp-bool-tv}, to prove $\Set$ is Boolean it suffices to prove it is two-valued.
  Thus, let $P$ be a subterminal object, i.e.\ a proposition, and let $S = \setof{ y\cin X | y=v \vee (y=u \wedge P)}$.
  Then $S$ is inhabited (by $v$), so it has a $\prec$-minimal element $y$.
  On the one hand, if $y=v$, then $\neg P$, for if $P$ then $u\in S$ and so $y$ would not be $\prec$-minimal in $S$.
  But on the other hand, if $(y=u \wedge P)$, then $P$.
\end{proof}

Amusingly (and perhaps surprisingly), it follows that ``classical well-foundedness'' implies ``inductive'' well-foundedness even in the absence of classical logic:

\begin{cor}
  If $X$ is a graph in which every inhabited subset $S\csub X$ has a $\prec$-minimal element, then it is well-founded.
\end{cor}
\begin{proof}
  Let $S$ be inductive and $x\cin X$; we want to show $x\in S$.
  Since $S$ is inductive, it suffices to show that for any $y\prec x$ we have $y\in S$.
  But if $y\prec x$, then by \autoref{thm:clwf-lem} excluded middle holds; thus $X$ is well-founded, so $S=X$ and hence $x\in S$.
\end{proof}

We will say no more about classical well-foundedness.
The intuitionistically correct \autoref{defn:well-founded} of well-foundedness suffices to perform proofs
by induction on a well-founded graph: by proving that a given subset
is inductive, we conclude it is the whole graph.  In order to prove a
\emph{statement} by well-founded induction, we must apply separation
to the statement to turn it into a subset; thus either the statement
must be \ddo\ or we must have a stronger separation axiom.  We can
also use well-founded induction in more than one variable, since if
$X$ and $Y$ are well-founded then so is $X\times Y$.

We record some observations about well-founded graphs.

\begin{lem}\label{thm:wf-hered}
  Any subset of a well-founded graph is well-founded with the induced relation.
\end{lem}
\begin{proof}
  For any $Z\csub X$, if $S\csub Z$ is inductive in $Z$, then
  the Heyting implication $(Z\imp S) \subseteq X$ is inductive in $X$.
\end{proof}

\begin{lem}\label{thm:wf-no-cycles}
  If $X$ is a well-founded graph, then there does not exist any cyclic
  path $x = x_n \prec \cdots \prec x_0 = x$ in $X$.
\end{lem}
\begin{proof}
  If there were, then the Heyting complement of $\{ x_0, \dots,
  x_n\}\subseteq X$ would be inductive but not all of $X$.
\end{proof}

\begin{lem}\label{thm:root-decid}
  If $X$ is a well-founded graph and $x\cin X$, then $x\colon 1\to X/x$
  is a complemented subobject.  In particular, it is decidable whether
  or not a node of a well-founded \apg\ is the root.
\end{lem}
\begin{proof}
  By definition, every node $y\cin X/x$ admits some path to $x$, of
  length $n\cin N$ say.  Since $0$ is a complemented subobject of $N$,
  either $n=0$, in which case $y=x$, or $n>0$.  If $n>0$, then if
  $x=y$ there would be a cyclic path from $x$ to itself; hence in this
  case $x\neq y$.  Thus for all $y$, either $y=x$ or $y\neq x$, as
  desired.
\end{proof}

\begin{defn}
  We will write $X\csl x$ for the complement of $x$ in $X/x$, i.e.\
  the set of nodes admitting a path to $x$ of length $>0$.  In
  particular, $X\csl \star$ is the complement of the root.
\end{defn}

We also need our \apgs\ to satisfy
the axiom of extensionality.

\begin{defn}
  An graph $X$ is \textbf{extensional} if for any two
  nodes $x$ and $y$, if $z\prec x \Leftrightarrow z\prec y$ for all $z$,
  then $x=y$.
\end{defn}

For non-well-founded graphs, this definition would have to be
strengthened in one of various possible ways (see~\cite{aczel:afa}).

\begin{rmk}
  An equivalent characterization of the universe of extensional
  well-founded \apgs\ can be obtained by working with well-founded
  \emph{rigid trees} instead.  A \emph{tree} is an \apg\ in which
  every node $x$ admits a \emph{unique} path to the root, and it is
  \emph{rigid} if for any node $z$ and any children $x\prec z$ and
  $y\prec z$, if $X/x \cong X/y$, then $x=y$.
  Every extensional well-founded \apg\ $X$ has an ``unfolding'' into a
  well-founded rigid tree $X^t$, whose nodes are the paths in $X$, and
  conversely every well-founded rigid tree is the unfolding of some
  extensional well-founded \apg.  Rigid trees are used
  by~\cite{cole:cat-sets,mitchell:topoi-sets} and~\cite{mm:shv-gl},
  while extensional relations are used by~\cite{osius:cat-setth}
  and~\cite{ptj:topos-theory}, but they result in essentially
  equivalent theories.  (In the non-well-founded case, however,
  extensional relations seem more generally applicable than rigid
  trees; see~\cite{aczel:afa}.)
\end{rmk}

We now record some observations about extensionality and
well-foundedness.

\begin{defn}
  An \textbf{initial segment} of a graph $X$ is a subset $Y\csub
  X$ such that $x\prec y \in Y$ implies $x\in Y$.
\end{defn}

\begin{lem}\label{thm:ext-hered}
  Any initial segment of an extensional graph is extensional with the
  induced relation.\qed
\end{lem}

\begin{lem}\label{thm:wf-rigid}
  If $X$ is a well-founded extensional graph and $X/x \cong X/y$, then
  $x=y$.
\end{lem}
\begin{proof}
  We prove this by well-founded induction.  If $g\colon X/x\toiso
  X/y$, then for all $x'\prec x$ we have $X/x' \cong X/g(x')$, where
  $g(x')\prec y$; hence by induction $x' = g(x')$.  Similarly, we have
  $y' = g^{-1}(y')$ for all $y'\prec y$.  By extensionality, $x=y$.
\end{proof}

In the language of~\cite{aczel:afa}, \autoref{thm:wf-rigid} says that
well-founded extensional graphs are ``Finsler-extensional.''

\begin{lem}\label{thm:wf-rigid-global}
  Any automorphism of a well-founded extensional graph is the
  identity.  Therefore, any two parallel isomorphisms of well-founded
  extensional graphs are equal.
\end{lem}
\begin{proof}
  Let $f\colon X\toiso X$ be an isomorphism; we prove by well-founded
  induction that $f(x)=x$ for all $x\in X$.  But if $f(x')=x'$ for all
  $x'\prec x$, then extensionality immediately implies $f(x)=x$.
\end{proof}

\begin{defn}
  We write $\bbV(\Set)$ for the class of well-founded extensional
  \apgs\ in \Set.
\end{defn}

Our goal is now to show that $\bbV(\Set)$ is a model of a material set
theory.  We consider two \apgs\ to be \emph{equal}, i.e.\ to represent
the same material-set, when they are isomorphic.  And we model the
membership relation $\ordin$ in the expected way:

\begin{defn}
  If $X$ is an \apg, the children of its root are called its
  \textbf{members}.  We write $|X|$ for the set of members of $X$.  If
  $X$ and $Y$ are \apgs, we write $X\iin Y$ to mean that $X\cong Y/y$
  for some member $y\cin |Y|$.
\end{defn}

Here is our omnibus theorem.

\begin{thm}\label{thm:str->mat}
  Let \Set\ be a constructively well-pointed \Pi-pretopos with a \nno,
  and let \ph\ be a formula of material set theory.
  Then $\bbV(\Set)\ss\ph$ whenever any of the following holds.
  \begin{blist}
  \item \ph\ is the axiom of extensionality, empty set, pairing,
    union, exponentiation, infinity, foundation, transitive closures,
    Mostowski's principle, or an instance of \ddo-separation.
  \item \Set\ satisfies structural fullness and \ph\ is material fullness.
  \item \Set\ is a topos and \ph\ is the power set axiom.
  \item \Set\ satisfies structural separation and \ph\ is an instance
    of material separation.
  \item \Set\ satisfies structural collection and \ph\ is an instance
    of material collection.
  \item \Set\ satisfies replacement of contexts and \ph\ is an instance
    of material replacement.
  \item \Set\ is Boolean and \ph\ is an instance of \ddo-classical
    logic.
  \item \Set\ satisfies full classical logic and \ph\ is an instance
    of full classical logic.
  \item \Set\ satisfies induction and \ph\ is an instance of full
    induction.
  \item \Set\ satisfies the axiom of choice and \ph\ is the material
    axiom of choice.
  \item \Set\ satisfies extensional well-founded induction and \ph\ is
    an instance of set-induction.
  \end{blist}
\end{thm}

\begin{rmk}
  As in \autoref{rmk:translation}, we can
  also interpret this theorem as giving a translation of formulas in the
  language of material set theory into formulas in the language of
  structural set theory, and thereby obtain an equiconsistency result
  under weaker assumptions on the metatheory.
\end{rmk}

We will prove this theorem with a series of lemmas, but first we need
to introduce some auxiliary notions leading up to the construction of
\emph{extensional quotients}.

\begin{defn}
  Let $X$ and $Y$ be graphs.  A \textbf{simulation} from $X$ to $Y$ is
  a function $f\colon X\to Y$ such that
  \begin{enumerate}
  \item if $x'\prec x$, then $f(x')\prec f(x)$, and
  \item if $y\prec f(x)$, then there exists an $x'\prec x$ with $f(x')=y$.
  \end{enumerate}
  A \textbf{bisimulation} from $X$ to $Y$ is a relation $R\csub
  X\times Y$ such that both projections $R\to X$ and $R\to Y$ are
  simulations (where $R$ is considered as a full subgraph of $X\times
  Y$).  A bisimulation is \textbf{bi-entire} if $R\to X$ and $R\to Y$
  are surjective.
\end{defn}

Note that if $f\colon X\to Y$ is a simulation, then the relation
$(1,f)\colon X\to X\times Y$ is a bisimulation, which is bi-entire iff
$f$ is surjective.

An example of a simulation is the inclusion of an initial
segment.  The following lemma says that for extensional well-founded
graphs, these are the only simulations.

\begin{lem}\label{thm:sim-inj}
  If $X$ is well-founded and extensional, then any simulation $f\colon
  X\to Y$ is injective, and isomorphic to the inclusion of an initial
  segment.
\end{lem}
\begin{proof}
  We show by well-founded induction that $f(x_1)=f(x_2)$ implies
  $x_1=x_2$.  Suppose $f(x_1)=f(x_2)$; then for any $z_1\prec x_1$, we
  have $f(z_1)\prec f(x_1) = f(x_2)$, so since $f$ is a simulation
  there is a $z_2\prec x_2$ with $f(z_2)=f(z_1)$.  Hence $z_1=z_2$ by
  induction, so $z_1\prec x_2$.  Dually, $z_2\prec x_2$ implies
  $z_2\prec x_1$, so by extensionality $x_1=x_2$.  Finally, any
  injective simulation must be an initial segment.
\end{proof}

\begin{lem}\label{thm:bisim-wf}
  If $R$ is a bi-entire bisimulation from $X$ to $Y$, then $X$ is
  well-founded if and only if $Y$ is so.
\end{lem}
\begin{proof}
  Suppose $X$ is well-founded and $S\subseteq Y$ is inductive.  Let
  $T\csub X$ consist of those $x\cin X$ such that $R(x,y)$ implies
  $y\in S$; we show $T$ is inductive.  Suppose $x\cin X$ is such
  that $x'\prec x$ implies $x'\in T$, and that $R(x,y)$.  Then
  for any $y'\prec y$ there is an $x'\prec x$ with $R(x',y')$, whence
  $x'\in T$ and thus $y'\in S$.  So since $S$ is inductive, $y$ must
  be in $S$.  Thus $x\in T$, so $T$ is inductive.  Since $X$ is
  well-founded, $T=X$, and then since $R$ is bi-entire, $S=Y$; thus
  $Y$ is well-founded.
\end{proof}

\begin{lem}\label{thm:bisim-iso}
  Any bi-entire bisimulation between extensional well-founded graphs
  must be an isomorphism.
\end{lem}
\begin{proof}
  Let $R\csub X\times Y$ be such.  We show by well-founded induction
  that if $R(x,y_1)$ and $R(x,y_2)$, then $y_1=y_2$.  For if $z_1\prec
  y_1$, then since $R$ is a bisimulation, there is a $w\prec x$ with
  $R(w,z_1)$, and again since $R$ is a bisimulation, there is a
  $z_2\prec y_2$ with $R(w,z_2)$.  By induction, $z_1=z_2$, so that
  $z_1\prec y_2$ for any $z_1\prec y_1$.  By symmetry, for any
  $z_2\prec y_2$ we have $z_2\prec y_1$, hence by extensionality
  $y_1=y_2$.

  By symmetry, if $R(x_1,y)$ and $R(x_2,y)$ then $x_1=x_2$, so $R$ is
  functional in both directions.  Since it is also bi-entire, it must
  be an an isomorphism.
\end{proof}

\begin{lem}\label{thm:wfext-bisid}
  If $X$ is well-founded, then it is extensional if and only if every
  bisimulation from $X$ to $X$ is contained in the identity.
\end{lem}
\begin{proof}
  Suppose first $X$ is well-founded and extensional and $R$
  is a bisimulation with $R(x_1,x_2)$.  Then $R$ is a bi-entire
  bisimulation from $X/x_1$ to $X/x_2$ (by ordinary induction on
  length of paths), so by \autoref{thm:bisim-iso} it must be an
  isomorphism $X/x_1\cong X/x_2$.  But $X$ is well-founded and
  extensional, so by \autoref{thm:wf-rigid}, $x_1=x_2$.

  Now suppose that every bisimulation from $X$ to $X$ is contained in
  the identity, and also that $x,y\cin X$ are such that $z\prec x
  \Leftrightarrow z\prec y$ for all $z$.  Define $R(a,b)$ to hold if
  either $a=b$, or $a=x$ and $b=y$.  Then $R$ is a bisimulation, and
  if it is contained in the identity, then $x=y$; hence $X$ is
  extensional.
\end{proof}

In the language of~\cite{aczel:afa}, \autoref{thm:wfext-bisid} says
that for well-founded graphs, extensionality is equivalent to ``strong
extensionality.''

\begin{cor}\label{thm:parallel-sim}
  If $Y$ is well-founded and extensional, then any two simulations
  $f,g\colon X\toto Y$ are equal.
\end{cor}
\begin{proof}
  The image of $(f,g)\colon X\to Y\times Y$ is a bisimulation,
  hence contained in the identity.
\end{proof}

Therefore, well-founded extensional graphs and simulations form a
(large) preorder.  In the language of material set theory, this
preorder represents the partial order of transitive-sets and subset
inclusions.


\begin{lem}\label{thm:bisim-quot}
  If $X$ is a graph and $R$ is a bisimulation from $X$ to $X$ which is
  an equivalence relation, then its quotient $Y$ inherits a graph
  structure such that the quotient map $[-]\colon X\epi Y$ is a
  simulation.  Also, if $X$ is an \apg, then so is $Y$.
\end{lem}
\begin{proof}
  Define $\prec$ on $Y$ to be minimal such that $[-]$ preserves
  $\prec$, i.e.\ $y_1\prec y_2$ if there exist $x_1\prec x_2$ in $X$
  with $[x_1] = y_1$ and $[x_2]=y_2$.  Now suppose that $y\prec [x]$.
  By definition this means that there exist $z_1\prec z_2$ with
  $[z_1]=y$ and $[z_2] = [x]$, i.e.\ $R(z_2,x)$ holds.  But $R$ is a
  bisimulation, so there exists $x'\prec x$ with $R(z_1,x')$,
  i.e. $[x']=y$; hence $[-]$ is a simulation.  If $X$ is an \apg, we
  define the root of $Y$ to be $[\star]$; accessibility of $Y$ follows
  directly from that of $X$.
\end{proof}

Of course, by \autoref{thm:bisim-wf}, if $X$ is well-founded, then so
is the quotient $Y$.  This gives us a way to construct extensional
quotients in some cases.

\begin{lem}\label{thm:ext-quotient-strong}
  If \Set\ is either a topos or satisfies full separation, then for every well-founded relation $X$ there is a surjective simulation $X\epi \Xbar$ where $\Xbar$ is well-founded and extensional, and an \apg\ if $X$ is.
\end{lem}
\begin{proof}
  Note that the union of any family of bisimulations is a bisimulation.
  More specifically, if we have a relation $R \csub A\times X\times X$ such that for each $a\cin A$ the relation $a^* R \csub X\times X$ is a bisimulation, then its image $\pi_A R \csub X\times X$ is a bisimulation.
  If \Set\ is a topos, then we can take $A$ to be the set of \emph{all} bisimulations (a subset of $P(X\times X)$), obtaining thereby the \emph{largest} bisimulation.

  A largest bisimulation $R$, if it exists, is necessarily an equivalence relation, since the identity, $R\circ R$, and $R^{-1}$ are all bisimulations as well and hence contained in it.
  Moreover, if $\Xbar$ is its quotient as in \autoref{thm:bisim-quot}, then any bisimulation $S$ on $\Xbar$ induces by pullback a bisimulation on $X$, which is contained in $R$, hence $S$ is contained in the identity; thus by \autoref{thm:wfext-bisid} $Y$ is extensional.
  This completes the topos case.
  If instead \Set\ satisfies full separation, then we can instead construct the largest bisimulation as the set of all $(x,y)\cin X\times X$ such that there exists a bisimulation $R$ from $X$ to $X$ with $(x,y)\in R$, and proceed similarly.
\end{proof}

Without power sets or separation,
it seems impossible to construct all extensional quotients,
but we can still construct them in a useful amount of generality.

\begin{lem}\label{thm:ext-quotient}
  Let $n$ be a fixed \emph{external} natural number, let $X$ be a
  well-founded \apg, and assume that $X/x$ is extensional whenever $x$
  is a node that admits a path of length $n$ to the root.  Then there
  is an extensional well-founded \apg\ $\Xbar$ and a surjective
  simulation $q\colon X\to \Xbar$.
\end{lem}
\begin{proof}
  The proof is by external induction on $n$.  (We will only need this
  lemma for $n\le 3$, so the reader is encouraged not to worry too
  much about what this induction requires of the metatheory.)  The
  base case is easy: since the root $\star$ admits a path of length
  $0$ to itself and $X\cong X/\star$, we can take $\Xbar=X$.

  Now suppose the statement is true for some $n$, and let $X$ satisfy
  the hypothesis for $n+1$.  For any $k$, write $X_{k}$ for the set of
  nodes admitting a path of length $k$ to the root.  Let the relation
  $R$ on $X$ be defined by $R(x,y)$ if there exists an isomorphism
  $X/x \toiso X/y$.  (This can be constructed using \ddo-separation
  and local exponentials in \Set.)  Then $R$ is a bisimulation and an
  equivalence relation, so by \autoref{thm:bisim-quot} it has a
  quotient $Y$ which is again a well-founded \apg, with $[-] : X\to Y$ a simulation.

  We claim that $Y$ satisfies the hypothesis for $n$.  Let $y\in Y_n$
  and suppose that $y_1, y_2\in Y/y$ satisfy $z\prec y_1 \Iff z\prec
  y_2$; we must show $y_1=y_2$.  Applying the simulation property
  inductively, we have $y = [x]$ for $x\in X_n$, and $y_i = [x_i]$
  with $x_i \in X/x$ for $i=1,2$.

  By \autoref{thm:root-decid}, for
  each $i$ either $x_i = x$ or $x_i \in X/w_i$ for some $w_i\prec x$.
  If $x_1 = x$ and $x_2=x$, then of course $y_1 = [x_1]=[x_2]=y_2$.

  If $x_1 = x$ and $x_2\in X/w_2$ with $w_2\prec x$, then $[w_2] \prec [x] = [x_1] = y_1$.
  Thus, by assumption, also $[w_2]\prec y_2 = [x_2]$.
  So by the simulation property, there is an $x_3$ with $[x_3] = [w_2]$ and $x_3\prec x_2$.
  Now $x_3\prec x_2$ and $x_2\in X/w_2$ imply $x_3 \in X/w_2$, while by definition of the quotient, $[x_3] = [w_2]$ means that $X/x_3 \cong X/w_2$, which we can therefore rewrite as $(X/w_2)/x_3 \cong (X/w_2)/w_2$.
  But since $w_2\prec x \in X_n$, we have $w_2 \in X_{n+1}$, so by assumption $X/w_2$ is extensional; thus \autoref{thm:wf-rigid} implies that $x_3=w_2$.
  Since $x_3\prec x_2$, this means $w_2 \prec x_2$, contradicting $x_2 \in X/w_2$.\footnote{I am indebted to the referee for noticing a bug in my original handling of this case and suggesting this correction.}

  Hence the only possible case is when
  $x_i \in X/w_i$ with $w_i\prec x$ for both $i=1,2$.
  This implies that $w_i\in X_{n+1}$, so each $X/w_i$ is
  extensional, and thus so is each $X/x_i$ by \autoref{thm:ext-hered}.
  Therefore, by \autoref{thm:sim-inj}, the quotient map $[-]\colon
  X\to Y$ induces an isomorphism $X/x_i \cong Y/[x_i]= Y/y_i$.  But
  $z\prec y_1 \Iff z\prec y_2$ means that $Y\csl y_1 \cong Y\csl y_2$,
  and hence (by \autoref{thm:root-decid}) also $Y/y_1 \cong Y/y_2$;
  thus we also have $X/x_1 \cong X/x_2$.  Thus, by definition,
  $R(x_1,x_2)$, and so $y_1 = [x_1]=[x_2]=y_2$.

  We have shown that $Y$ satisfies the hypothesis for $n$.  Thus it
  has an extensional quotient $\Ybar$, and so the composite $X\to Y\to
  \Ybar$ is an extensional quotient of $X$.
\end{proof}

This also yields another hypothesis under which \emph{all} extensional quotients exist.

\begin{cor}
  If \Set\ satisfies well-founded induction and replacement of contexts, then for any well-founded \apg\ $X$ there is an extensional well-founded \apg\ $\Xbar$ and a surjective simulation $q\colon X\to \Xbar$.
\end{cor}
\begin{proof}
  For any well-founded \apg\ $X$, we prove by induction on $x\cin X$ that the \apg\ $X/x$ admits a surjective simulation $X/x\epi Y$ where $Y$ is well-founded and extensional.
  (This requires the axiom of well-founded induction since the existence of such a $Y$ is an unbounded quantifier.)
  The case $x=\star$ then yields the result.

  Thus, suppose this is true for all $y\prec x$, and let $(X/-)$ be the transitive closure of $\prec$ restricted to $X\times |X/x|$, regarded as an object of $\Set/|X/x|$, so that $y^*(X/-) \cong X/y$ for any $y\prec x$.
  If $\Set$ satisfied collection, we could find a surjection $f:A\epi |X/x|$, a pointed graph $p:Y\to A$ in $\Set/A$ (here ``pointed'' means we have a
  section $s\colon A\to Y$ over $A$), and a function $g:f^*(X/-)\to Y$ in $\Set/A$, such that for any $a\cin A$, the pullback $a^*(g)$ is a surjective simulation $X/f(a) \epi a^*Y$ where $a^*Y$ is well-founded and extensional.
  But in fact replacement of contexts is sufficient for this (and we can take $A = |X/x|$), because given $X$, the context extension consisting of an extensional well-founded \apg\ $\Xbar$ and a surjective simulation $q\colon X\to \Xbar$ is unique up to unique isomorphism if it exists by \autoref{thm:bisim-iso} and \autoref{thm:parallel-sim}.\footnote{I am indebted to the referee for pointing this out.}
  Now give $Y+1$ the relation $\prec$ induced from $Y$ with also $s(a)\prec \star$ for all $a\cin A$, where $\star$ is the new point added.
  Then $Y+1$ satisfies the hypotheses of \autoref{thm:ext-quotient} with $n=1$; let $Z$ be its extensional quotient.

  Define $h:X/x\to Z$ by sending $x$ to the image of $\star \cin Y+1$, and sending $z\in X\csl x$ to its image by $g_a$, where $a\cin A$ is such that $z\prec f(a)$.
  This is well-defined since $g$ then restricts to a surjective simulation $X/z \epi Z/h(z)$, and there can be only one $h(z)\cin Z$ with this property by extensionality of $Z$.
\end{proof}

We can now start verifying the axioms of material set theory.

\begin{lem}
  The axiom of extensionality holds.  That is, two well-founded
  extensional \apgs\ $X$ and $Y$ are isomorphic iff $Z\iin X \Iff
  Z\iin Y$ for all $Z$.
\end{lem}
\begin{proof}
  The ``only if'' direction is clear, so suppose that $Z\iin X \Iff
  Z\iin Y$ for all $Z$.  Then for every $x\cin |X|$, we have $X/x \iin
  Y$, hence $X/x \cong Y/y$ for some $y\cin Y$.  By
  \autoref{thm:wf-rigid}, this $y$ must be unique, and conversely as
  well; hence we have a bijection $g\colon |X|\toiso |Y|$ such that
  for any $x\cin |X|$ there exists an isomorphism $h_x\colon X/x \toiso
  Y/g(x)$ (which must be unique, by \autoref{thm:wf-rigid-global}).
  Define a relation $R$ from $X$ to $Y$ such that $R(a,b)$ holds if:
  \begin{enumerate}
  \item $a=\star$ and $b=\star$, or
  \item there exists $x\cin |X|$ such that $a\in X/x$, $b\in Y/g(x)$,
    and $h_x(a)=b$.
  \end{enumerate}
  Then $R$ is a bi-entire bisimulation, so by \autoref{thm:bisim-iso} it
  is an isomorphism.
\end{proof}

\begin{lem}\label{thm:struct-emppairun}
  The axioms of empty set, pairing, and union hold.
\end{lem}
\begin{proof}
  The empty set is represented by the \apg\ with one node and no
  $\prec$ relations, which has no members.

  If $X$ and $Y$ are extensional well-founded \apgs, let $Z = X + Y +
  1$ with $\prec$ induced from $X$ and $Y$ along with $\star_X \prec
  \star$ and $\star_Y\prec \star$, where $\star$ is the new point
  added.  Since $X$ and $Y$ are extensional, $Z$ satisfies the
  hypothesis of \autoref{thm:ext-quotient} with $n=1$, and its
  extensional quotient represents the pair $\{X,Y\}$.

  Finally, if $X$ is an extensional well-founded \apg, let $\Vert X
  \Vert$ denote the subset of those $x\cin X$ such that $x\prec y\prec
  \star$ for some $y$, let $Y$ be the subset of $X$ consisting of
  those nodes admitting some path to a node in $\Vert X \Vert$, and
  define $Z = Y + 1$ with \prec\ inherited from $Y$ and with $y\prec
  \star$ for each $y\in \Vert X\Vert \subseteq Y$, where $\star$ is
  the new element added.  Then $Z$ satisfies the hypotheses of
  \autoref{thm:ext-quotient} with $n=1$, so it has an extensional
  quotient, which is the desired union $\bigcup X$.
\end{proof}

\begin{lem}\label{thm:cart-prod}
  Cartesian products (using Kuratowski ordered pairs) exist in
  $\bbV(\Set)$.
\end{lem}
\begin{proof}
  Let $X$ and $Y$ be extensional well-founded \apgs, and consider the set
  \[ Z =
  (X\csl\star) + (Y\csl\star) +
  |X| + (|X|\times |Y|) +
  (|X|\times |Y|)
  + 1.
  \]
  For $x\cin |X|$ we write $x$ for its image in $(X\csl\star)\mono Z$ and $x'$
  for its image in $|X|\mono Z$.  Similarly, we write $y$ for
  images in $Y\csl\star$, $(x,y)$ for images in the first copy of
  $|X|\times |Y|$, $(x,y)'$ for images in the second copy, and $\star$
  for the final point.  We define $\prec$ on $Z$ as follows:
  \begin{blist}
  \item $\prec$ on $X\csl\star$ and $Y\csl\star$ is induced from
    $X$ and $Y$.
  \item $x\prec x'$ for all $x\cin |X|$.
  \item $x\prec (x,y)$ and $y\prec (x,y)$ for all $y\cin |Y|$.
  \item $x' \prec (x,y)'$ and $(x,y)\prec (x,y)'$ for all $x\cin |X|$
    and $y\cin |Y|$.
  \item $(x,y)'\prec \star$ for all $x\cin |X|$ and $y\cin |Y|$.
  \end{blist}
  It is straightforward to verify that $Z$ is then a well-founded
  \apg.  Since $X$ and $Y$ are extensional, $Z$ satisfies the
  hypothesis of \autoref{thm:ext-quotient} with $n=3$.  Its
  extensional quotient then represents the cartesian product of $X$
  and $Y$.
\end{proof}

Henceforth, we will write $X\otimes Y$ for the extensional well-founded \apg\ constructed above, which represents the cartesian product of the material sets $X$ and $Y$.
We will also reuse the notations such as $(x,y)'$ for the images of these elements in $X\otimes Y$.

\begin{lem}\label{thm:struc-exp}
  $\bbV(\Set)$ satisfies the exponentiation axiom.
\end{lem}
\begin{proof}
  If $X$ and $Y$ are extensional well-founded \apgs, define
  \[W =
  ((X\otimes Y)\csl \star) +
  |Y|^{|X|} +
  1
  \]
  with $\prec$ induced from $X\otimes Y$, along with $(x,y)' \prec f$ whenever $f(x)=y$ and
  $f\prec \star$ for any $f\in {|Y|}^{|X|}$.  Then $W$ is a
  well-founded \apg, which is in fact already extensional; we claim it represents the material function-set.

  It is clear that if $F\iin W$, then $F\in \bbV(\Set)$ is a function
  from $X$ to $Y$ in the sense of material set theory.  Conversely,
  from any $F\in \bbV(\Set)$ which is a function from $X$ to $Y$,
  consider the subset of $|X|\times |Y|$ determined by those $(x,y)$
  such that the Kuratowski ordered pair $\{ \{ X/x\}, \{X/x,Y/y\}\}$
  is $\iin F$.  This defines a function $|X|\to |Y|$ in \Set, which
  therefore induces an $f\cin |W|$ such that $F\cong W/f$.
\end{proof}

Observe that in particular, we have shown that for $X,Y\in\bbV(\Set)$,
there is a 1-1 correspondence between functions $|X|\to |Y|$ in \Set\
and isomorphism classes of \apgs\ $F\in\bbV(\Set)$ which represent
functions from $X$ to $Y$ in the sense of material set theory.  In
fact, although $\bbV(\Set)$ need not satisfy limited \ddo-replacement,
so that \autoref{thm:iz-topos} does not apply to it directly, we still
have:

\begin{lem}\label{thm:struc-mat-set}
  The sets and functions in $\bbV(\Set)$ form a category
  $\bbSet(\bbV(\Set))$, which can be identified with a full
  subcategory of $\Set$ closed under finite limits, subsets,
  quotients, and local exponentials.  In particular,
  $\bbSet(\bbV(\Set))$ is a \Pi-pretopos.
\end{lem}
\begin{proof}
  We have observed closure under products and non-local exponentials.
  Closure under subsets is easy: if $U\csub |X|$, then the
  sub-graph $Y$ of $X$ consisting of the root and all nodes admitting
  a path to $U$ is a well-founded extensional \apg\ with $|Y|\cong U$.
  This then implies closure under finite limits.

  For quotients, if $R$ is an equivalence relation on $|X|$, let
  $q\colon |X|\to Y$ be the quotient of $|X|$ by $R$ in \Set; then the
  \apg\
  \[ Z = (X\csl\star) + Y + 1,\]
  with $\prec$ inherited from $X$ along with $x\prec q(x)$ and $y\prec
  \star$ for all $x\cin |X|$ and $y\cin Y$, is well-founded and
  extensional and has $|Z|\cong Y$.

  Finally, given $f\colon |A|\to |B|$ and $g\colon |X|\to |A|$, we
  have the local exponential $h\colon \Pi_f(g) \to |B|$ in \Set, with
  counit $e\colon \Pi_f(g)\times_{|B|} |A| \to |X|$.  Let $Z$
  represent the material cartesian product of $A$ and $X$ as in
  \autoref{thm:cart-prod}, and consider the \apg\
  \[ W= ((X\otimes Y)\csl\star) + \Pi_f(g) + 1 \]
  with $\prec$ inherited from $X\otimes Y$, along with $(x,y)' \prec j$
  whenever $f(x) = h(j)$ and $e(j,x)=y$, and $j\prec \star$ for all
  $j\in \Pi_f(g)$.  Then $W$ is well-founded and extensional and satisfies
  $|W|\cong \Pi_f(g)$; hence $\bbSet(\bbV(\Set))$ is closed under
  local exponentials.
\end{proof}



\begin{lem}
  If \Set\ satisfies structural fullness, $\bbV(\Set)$ satisfies
  material fullness.
\end{lem}
\begin{proof}
  Analogously to exponentiation, if $R\csub M\times X\times Y$ is a
  generic set of multi-valued functions from $|X|$ to $|Y|$, consider
  \[W = ((X\otimes Y)\csl\star) + M + 1
  \]
  with $\prec$ induced from $X\otimes Y$ along with $(x,y)'\prec m$ if
  $R(m,x,y)$ and $m\prec \star$ for all $m\cin M$.  Then $W$ is a
  well-founded \apg\ and satisfies the hypothesis of
  \autoref{thm:ext-quotient} with $n=1$.  Its extensional quotient
  represents a generic set of multi-valued functions, by a similar
  argument as for exponentiation.
\end{proof}

\begin{lem}\label{thm:struc-pow}
  If \Set\ is a topos, then $\bbV(\Set)$ satisfies the power set
  axiom, and the inclusion $\bbSet(\bbV(\Set)) \hookrightarrow \Set$
  is a logical functor.
\end{lem}
\begin{proof}
  For an extensional well-founded \apg\ $X$, define
  \[ Y = (X\csl \star) + P |X| + \star
  \]
  with $\prec$ induced from $X$ along with $x\prec A$ whenever $x\cin
  |X|$, $A\cin P |X|$, and $x\in A$; and also of course $A\prec
  \star$.  This is an extensional well-founded \apg\ that represents
  the material power set of $X$.  The second statement is immediate.
\end{proof}

\begin{lem}\label{thm:inf}
  $\bbV(\Set)$ satisfies the axiom of infinity.
\end{lem}
\begin{proof}
  Let $\omega = N + 1$, with $\prec$ defined to be $<$ on $N$ together
  with $n\prec \star$ for all $n\in N$.  This is an extensional
  well-founded \apg\ which represents the von Neumann ordinal \omega.
  The infinity axiom follows from the universal property of the \nno.
\end{proof}

\begin{lem}\label{thm:struc-ac}
  If \Set\ satisfies the axiom of choice, then so does $\bbV(\Set)$.
\end{lem}
\begin{proof}
  Let $X$ be a well-founded extensional \apg\ such that $X\csl x$ is
  inhabited for each $x\cin |X|$.  If $Y\csub X\times |X|$ consists
  of those $(y,x)$ such that $y\in X\csl x$, then the projection $Y\to
  |X|$ is surjective, and hence has a section, say $s$.  From $s$ we
  can easily construct a material choice function for $X$.
\end{proof}

I do not know any way to show that having enough projectives in \Set\ implies the presentation axiom for $\bbV(\Set)$; the problem is that a projective object in \Set\ seemingly need not admit any structure of a extensional well-founded \apg.

\begin{lem}\label{thm:struc-tc}
  $\bbV(\Set)$ has transitive closures.
\end{lem}
\begin{proof}
  If $X$ is a well-founded extensional \apg, let $T = (X\csl \star) +
  1$ with $\prec$ inherited from $X$ along with $x\prec \star$ for all
  nodes $x\cin X$ (not just all members).  Then $T$ is a well-founded
  extensional \apg\ which represents the transitive closure of $X$.
\end{proof}

\begin{lem}\label{thm:struc-most}
  $\bbV(\Set)$ satisfies Mostowski's principle.
\end{lem}
\begin{proof}
  Since $\bbSet(\bbV(\Set))$ is closed in \Set\ under finite limits
  and subsets, any well-founded extensional graph $X$ constructed in
  the material set theory $\bbV(\Set)$ will induce such a graph in
  $\Set$.  Therefore, $X+1$, with $\prec$ induced from $X$ along with
  $x\prec \star$ for all nodes $x\cin X$, is a well-founded extensional
  \apg, i.e.\ an object of $\bbV(\Set)$.  We then verify that it is
  transitive and isomorphic to $X$ in $\bbV(\Set)$.
\end{proof}

We now turn to the axiom schemata,
for which we need to translate material formulas in $\bbV(\Set)$ into structural ones in \Set.
We have already been doing this informally, but for the axiom schemata it is worth being more formal.
First of all, we translate a material \emph{context} \Gm\ (which is just a list of variables) into a structural context $\Gm_{\ordiin}$ by replacing each variable with a context that specifies the underlying data of an \apg: objects $X$ and $R$ and arrows $R \to X\times X$ and $\star\colon 1\to X$.
If $\rho\colon \Gm\to\bbV(\Set)$ is a assignment of material variables to extensional well-founded \apg{}s in \Set, there is an obvious corresponding assignment $\rho_{\ordiin}\colon \Gm_{\ordiin} \to \Set$.

Now, if \ph\ is a material formula in a context \Gm, we define a structural formula $\ph_{\ordiin}$ in context $\Gm_{\ordiin}$ follows:
\begin{blist}
\item We replace the material ``equality'' symbol $=$ by isomorphism $\cong$ of \apgs.
\item We replace the material ``membership'' symbol $\in$ by the relation $\iin$.
\item The connectives are unchanged.
\item We replace quantifiers over material-sets by quantifiers over well-founded extensional \apgs.
  For example, $\exists x. \ph(x)$ becomes \qq{there exists a well-founded extensional \apg\ $X$ such that $\ph_{\ordiin}(X)$}.
\end{blist}
It is straightforward to see that $\bbV(\Set) \ss_\rho \ph$ if and only if $\Set \ss_{\rho_{\ordiin}} \ph_{\ordiin}$.
This translation works quite well for the schemata involving arbitrary formulas.

\begin{lem}\label{thm:struc-sep}
  If \Set\ satisfies separation, then so does $\bbV(\Set)$.
\end{lem}
\begin{proof}
  Let $\ph(x)$ be a formula and $A\in\bbV(\Set)$.
  Using separation in \Set, let $U \csub |A|$ consist of precisely those $a\cin A$ such that $\ph_{\ordiin}(A/a)$.
  Define $B$ to consist of the root of $A$ together with all nodes admitting a path to some node in $U$.
  Then $B$ is a well-founded extensional \apg, and for any $C\iin A$ we have $C\iin B$ iff $\ph_{\ordiin}(C)$.
\end{proof}

\begin{lem}\label{thm:struc-fullbool}
  If \Set\ satisfies full classical logic, then so does $\bbV(\Set)$.
\end{lem}
\begin{proof}
  Classical logic for \Set\ implies $\ph_{\ordiin} \vee \neg \ph_{\ordiin}$ for any formula $\ph$ in $\bbV(\Set)$.
\end{proof}

\begin{lem}\label{thm:struc-coll}
  If \Set\ satisfies collection, then so does $\bbV(\Set)$.
\end{lem}
\begin{proof}
  Suppose that $A\in\bbV(\Set)$ and that \ph\ is a formula such that
  for any $X\iin A$, there exists a $Y\in\bbV(\Set)$ with $\ph(X,Y)$.
  This means that for any $x\cin |A|$, there exists a well-founded
  extensional \apg\ $Y$ such that $\ph_{\ordiin}(A/x,Y)$.  By
  collection in \Set, there is a surjection $V\xepi{p} |A|$ and a
  pointed graph $B$ in $\Set/V$ (where ``pointed'' means we have a
  section $s\colon V\to B$ over $V$) such that for each $v\cin V$,
  $v^*B$ is a well-founded extensional \apg\ and
  $\ph_{\ordiin}(A/p(v),v^*B)$.  It is easy to show that $B$,
  considered as a graph in \Set, is still well-founded.  And since $B$
  is a graph in $\Set/V$, its relation $\prec$ is fiberwise; thus for
  each $v\cin V$ we have $B/s(v) \cong v^*B$, which is therefore
  extensional and accessible.  It follows that $B+1$, with $s(v)\prec
  \star$ for all $v\in V$, satisfies the hypotheses of
  \autoref{thm:ext-quotient} with $n=1$.  Its extensional quotient is
  then the set desired by the material collection axiom.
\end{proof}

\begin{lem}\label{thm:struc-rep}
  If \Set\ satisfies replacement of contexts, then $\bbV(\Set)$ satisfies material replacement.
\end{lem}
\begin{proof}
  Suppose that $A\in\bbV(\Set)$ and that \ph\ is a formula such that
  for any $X\iin A$, there exists a unique $Y\in\bbV(\Set)$ with
  $\ph(X,Y)$.  Thus, for any $x\cin |A|$, there exists a well-founded
  extensional \apg\ $Y$ such that $\ph_{\ordiin}(A/x,Y)$, and any two
  such $Y$ are isomorphic.  Since any such isomorphism is unique by
  \autoref{thm:wf-rigid-global}, replacement of contexts in \Set\
  supplies a pointed graph $B$ in $\Set/|A|$ such that for each $x\cin
  |A|$ we have $\ph_{\ordiin}(A/x,x^*B)$.  The extensional quotient of
  $B+1$ then represents the set desired by material replacement.
\end{proof}

\begin{lem}\label{thm:struc-ind}
  If \Set\ satisfies full induction, then so does $\bbV(\Set)$.
\end{lem}
\begin{proof}
  If $\ph(x)$ is as in the statement of the material induction axiom,
  then $\ph_{\ordiin}(X)$ is a statement about some $X\iin \omega$,
  i.e.\ $X\iin N$.  But every $X\iin N$ is isomorphic to $N/n$ for a
  unique $n\cin N$, so $\ph_{\ordiin}(X)$ is equivalent to a statement
  $\ph_{\ordiin}'(n)$ about some $n\cin N$, which can then be proven by
  the structural induction axiom.
\end{proof}

\begin{lem}\label{thm:struc-setind}
  If \Set\ satisfies extensional well-founded induction, then $\bbV(\Set)$ satisfies set-induction.
\end{lem}
\begin{proof}
  Just like \autoref{thm:struc-ind}, using induction over the
  extensional well-founded relation in \Set\ that underlies any
  object of $\bbV(\Set)$.
\end{proof}

For the schemata involving \ddo-formulas, we require a different translation.
Suppose \ph\ is a \ddo-formula in a material context $\Gm=(x_1,\dots,x_k)$.
Let $\Th$ denote the structural context $(U, R, i_1\colon R\to U, i_2\colon R\to U)$, regarded as the data of a set with a binary relation which we denote $\prec$.
We define a structural \ddo-formula $\ph_{\ordiin}^{\ddo}$ in the context $\Gm_{\ordiin}^{\ddo} = (\Th, x_1,\dots,x_k)$, where each $x_i$ is now regarded as a \ddo-variable $1\to U$.
The translation is as follows.
\begin{blist}
\item Each variable $x_i$ is unchanged.
\item We replace the symbol $=$ by equality of \ddo-terms $1\to U$.
\item We replace the symbol $\in$ by the relation $\prec$ on elements of $U$.
\item The connectives are unchanged.
\item A \ddo-quantifier of the form $\exists x\in y.\ph$ is replaced by a \ddo-quantifier of the form $\exists x\cin U.(x\prec y\meet \ph)$.
  Similarly, $\forall x\in y.\ph$ is replaced by $\forall x\cin U.(x\prec y\imp \ph)$.
\end{blist}
Now let $\rho\colon \Gm \to \bbV(\Set)$ be a variable assignment in material set theory; thus it supplies well-founded extensional \apg{}s $\rho(x_1),\dots,\rho(x_k)$ in \Set.
Then
\[\rho(x_1) + \dots + \rho(x_k) + 1,
\]
with $\prec$ inherited from the $A_i$ along with $\star_i \prec \star$ for all $1\le i\le k$, satisfies the hypothesis of \autoref{thm:ext-quotient} with $n=1$.
Let $T$ denote its extensional quotient, and let $\rho_{\ordiin}^{\ddo}\colon \Gm_{\ordiin}^{\ddo} \to \Set$ be the structural context which assigns $U$ to $T$ with its induced binary relation $\prec$, and each $x_i$ to $[\star_i]$.

\begin{lem}
  We have $\bbV(\Set) \ss_{\rho}\ph$ if and only if $\bbV(\Set) \ss_{\rho_{\ordiin}^{\ddo}} \ph_{\ordiin}^{\ddo}$.
\end{lem}
\begin{proof}
  \autoref{thm:wf-rigid} implies that $T/x \cong T/y$ if and only if $x=y$.
  Similarly, if $T/x \iin T/y$, then $T/x \cong (T/y)/y' = T/y'$ for some $y'\prec y$, whence $x=y'$ and so $x\prec y$.
  The converse is easy, so $T/x \iin T/y$ if and only if $x\prec y$.
  Thus the atomic formulas correspond, and the connectives evidently do, so it remains to observe that \ddo-quantifiers are adequately represented by quantifiers over $T$, since by definition $X\iin T/y$ if and only if $X\cong T/x$ for some $x\prec y$.
\end{proof}

\begin{lem}\label{thm:struc-ddosep}
  $\bbV(\Set)$ satisfies \ddo-separation.
\end{lem}
\begin{proof}
  Just like \autoref{thm:struc-sep}, but using $\ph_{\ordiin}^{\ddo}$ instead of $\ph_{\ordiin}$, and the \ddo-separation property of \autoref{thm:ddo-sep} instead of the separation axiom.
\end{proof}

\begin{lem}\label{thm:struc-fdn}
  $\bbV(\Set)$ satisfies foundation.
\end{lem}
\begin{proof}
  For any \ddo-formula $\ph$ in $\bbV(\Set)$ and any well-founded
  \apg\ $X\in\bbV(\Set)$, we can form $\setof{ x\cin X | \ph(X/x) }$
  using \ddo-separation, since $\ph(X/x) \Leftrightarrow
  \ph_{\ordiin}^{\ddo}(x)$.  The hypothesis on \ph\ implies that this
  is an inductive subset of $X$, hence all of it.  This implies the
  desired conclusion, since for every extensional well-founded \apg\
  $Y$ there is another one $X$ with $Y\iin X$.
\end{proof}

\begin{lem}\label{thm:struc-bool}
  If \Set\ is Boolean, then $\bbV(\Set)$ satisfies \ddo-classical logic.
\end{lem}
\begin{proof}
  For any \ddo-formula \ph\ in $\bbV(\Set)$, \ddo-separation supplies
  a subset $\setof{ \emptyset | \ph }\subseteq \{\emptyset\}$.  If
  this is complemented in \Set, then we must have $\ph\vee\neg\ph$.
\end{proof}

This completes the proof of \autoref{thm:str->mat}.
We deduce the following.

\begin{cor}
  ETCS and BZC are equiconsistent.
\end{cor}

\begin{cor}
  Intuitionistic ETCS is equiconsistent with Intuitionistic BZ.
\end{cor}

\begin{cor}
  ETCS plus structural collection, and ETCS plus structural replacement (of contexts), are both equiconsistent with ZFC.
\end{cor}

\begin{cor}
  Intuitionistic ETCS plus structural collection is equiconsistent with IZF.
\end{cor}

So far we have concentrated on building models of pure sets only.
However, we can also allow an arbitrary set of atoms: we fix some
$A\in \Set$ and modify our definitions as follows.  (The definitions
not listed below need no modification.)
\begin{blist}
\item An \textbf{$A$-graph} is a graph $X$ together with a partial
  function $\ell\colon X \rightharpoonup A$, such that
  $\mathrm{dom}(\ell)$ is a complemented subobject of $X$, and if
  $x\prec y$ then $y\notin \mathrm{dom}(\ell)$.
\item Isomorphisms between $A$-graphs must
  preserve the labeling functions $\ell$.
\item An $A$-graph is \textbf{$A$-extensional} if
  \begin{enumerate}
  \item $\ell$ is injective on its domain, and
  \item for any $x,y\notin\mathrm{dom}(\ell)$, if $z\prec x
    \Leftrightarrow z\prec y$ for all $z$ then $x=y$.
  \end{enumerate}
\item An \textbf{$A$-simulation} between $A$-graphs is a simulation
  $f\colon X\to Y$ such that for any $x\cin X$, if $\ell(x)$ or
  $\ell(f(x))$ is defined, then both are and they are equal.
\end{blist}
Note that if an $A$-graph $X$ is an \apg\ and
$\star\in\mathrm{dom}(\ell)$, then $X$ can have no nodes other than
the root, so $X$ is simply an element of $A$.  Thus, accessible
pointed $A$-graphs model atoms themselves in addition to sets that can
contain atoms.

All the above lemmas, hence also \autoref{thm:str->mat}, still hold when extensional well-founded \apgs\ are replaced by $A$-extensional well-founded accessible pointed $A$-graphs.
The resulting material set theories satisfy the modified axioms for theories with atoms listed at the end of \S\ref{sec:sets}, including decidability of sethood and the existence of a set of atoms (namely, $A$), and do allow sethood in axiom schemas.

\begin{rmk}
  The ability to use any object $A$ of our category \Set\ as ``the set of atoms'' means that when working with any fixed finite collection of objects in \Set, we are free to use material set-theoretic language.
  However, there seems to be no structural way to construct a \emph{single} model of material set theory which contains the elements of \emph{all} the objects of \Set\ as atoms, since without a way to compare objects for equality, we cannot compare their elements for equality.
  In~\cite{abss:long-version}, it is shown that this is the only obstruction: if we do have a way to interpret equality on objects (``internally'' to \Set), then such a model can be constructed.
\end{rmk}

\section{Material-structural equivalences}
\label{sec:conclusion}

Of course, it is natural to ask to what extent the constructions
$\bbV(-)$ and $\bbSet(-)$ are inverse.  We have already noticed the
canonical inclusion $\bbSet(\bbV(\Set))\into \Set$, and the following
is easy to verify.

\begin{lem}
  The inclusion $\bbSet(\bbV(\Set))\into \Set$ is an equivalence if
  and only if every object of \Set\ can be embedded into some
  extensional well-founded graph.\qed
\end{lem}

I propose to call this property of \Set\ the \emph{axiom of
  well-founded materialization}.  (Other axioms of materialization
would arise from using various kinds of non-well-founded graphs.)

Note that well-founded materialization follows from the axiom of
choice, since then every object can be well-ordered and thus given the
structure of an \apg\ representing a von Neumann ordinal.  We remark
in passing that this implies that ``all replacement schemata for ETCS
are equivalent.''

\begin{prop}\label{thm:etcs-rep}
  Let $\mathcal{T}_1$ and $\mathcal{T}_2$ be axiom schemata for
  structural set theory such that for each $i=1,2$,
  \begin{enumerate}
  \item if \bV\ satisfies ZFC, then $\bbSet(\bV)$ satisfies
    $\mathcal{T}_i$, and\label{item:er1}
  \item if \Set\ satisfies ETCS+$\mathcal{T}_i$, then $\bbV(\Set)$
    satisfies ZFC.\label{item:er2}
  \end{enumerate}
  Then $\cT_1$ and $\cT_2$ are equivalent over ETCS.
\end{prop}
\begin{proof}
  Since ETCS includes AC, for any model of ETCS we have $\Set \simeq
  \bbSet(\bbV(\Set))$.  Thus, if $\Set\ss\cT_1$, then
  by~\ref{item:er2} we have $\bbV(\Set)\ss$ ZFC, hence
  by~\ref{item:er1} we have $\Set = \bbSet(\bbV(\Set))\ss\cT_2$.  The
  converse is the same.
\end{proof}

We have seen that our axioms of collection and replacement from
\S\ref{sec:strong-ax} satisfy~\ref{item:er1} and~\ref{item:er2}, so
over ETCS they are equivalent to any other such schema.  This includes
the reflection axiom of Lawvere~\cite[Remark 12]{lawvere:etcs}, the
axiom CRS of Cole~\cite{cole:cat-sets}, the axiom RepT of
Osius~\cite{osius:cat-setth}, and the replacement axiom of
McLarty~\cite{mclarty:catstruct}.  However, our axioms seem to be more
appropriate in an intuitionistic context.

Returning to the material-structural comparison, on the other side we
have:

\begin{prop}\label{thm:fdn-eqv}
  If \bV\ satisfies the core axioms of material set theory along with
  infinity, exponentials, foundation, and transitive closures, then
  there is a canonical embedding $\bV \to \bbV(\bbSet(\bV))$.  This
  map is an isomorphism if and only if \bV\ additionally satisfies
  Mostowski's principle.
\end{prop}
\begin{proof}
  By assumption, any $x\in \bV$ has a transitive closure
  $\mathrm{TC}(x)$, which is a well-founded extensional graph.  If we
  define $Y = \mathrm{TC}(x)+1$, with $\prec$ induced by $\in$ on
  $\mathrm{TC}(x)$ and with $z\prec \star$ for all $z\in x$, then $Y$
  is a well-founded extensional \apg, i.e.\ an object of
  $\bbV(\bbSet(\bV))$.  This construction gives a map $\bV \to
  \bbV(\bbSet(\bV))$, and it is straightforward to verify that it
  preserves and reflects membership and equality.

  Now if this embedding is an isomorphism, then clearly \bV\ must
  satisfy Mostowski's principle, since $\bbV(\bbSet(\bV))$ does so.
  Conversely, if \bV\ satisfies Mostowski's principle, then every
  $X\in \bbV(\bbSet(\bV))$ is isomorphic in \bV\ to a transitive set,
  and therefore equal in $\bbV(\bbSet(\bV))$ to something in the image
  of the embedding.
\end{proof}

We say that a material and a structural set theory are ``equivalent'' if we have a bijection between
\begin{enumerate}
\item isomorphism classes of models of the material set theory and
\item equivalence classes of models of the structural set theory.
\end{enumerate}

\begin{cor}
  ETCS is equivalent to MOST (i.e.\ BZC plus transitive closures and Mostowski's principle).
\end{cor}

\begin{cor}
  ETCS plus structural collection (or replacement) is equivalent to ZFC.
\end{cor}

\begin{cor}
  Intuitionistic ETCS plus well-founded materialization is equivalent to Intuitionistic BZ plus transitive closures and Mostowski's principle (which we may call ``Intuitionistic MOST'').
\end{cor}

\begin{cor}
  Intuitionistic ETCS plus structural collection and well-founded materialization is equivalent to IZF.
\end{cor}

\begin{rmk}
  The theory M$_0$ of~\cite{mathias:str-maclane} consists of the core
  axioms together with power sets and full classical logic.  If
  $\bV\ss$ M$_0$, then $\bbV(\bbSet(\bV))$ is precisely the model
  $W_1$ constructed in~\cite[\S2]{mathias:str-maclane}, which
  satisfies M$_0$ plus regularity, transitive closures, and
  Mostowski's principle, and inherits infinity and choice from \bV.
  (In the presence of classical logic and power sets, the axiom of
  infinity is unnecessary for the construction $\bbV(-)$.)
  In particular, if $\bV\ss$ ZBQC, then $\bbV(\bbSet(\bV))\ss$ MOST,
  while $\bbV(\bbSet(\bV))\cong \bV$ if we already had $\bV\ss$ MOST.
\end{rmk}

\begin{rmk}
  Mitchell~\cite{mitchell:booltopoi} observes that for Boolean topoi,
  the constructions we have called $\bbV(-)$ and $\bbSet(-)$ can
  be made into a pair of adjoint functors, modulo some 2-categorical technicalities not addressed therein.
  I expect this is also true
  in a constructive context.
  Moreover, the adjunction should be idempotent,
  and thus induce the above equivalences between
  the images of the two functors.
\end{rmk}

\bibliographystyle{halpha}
\bibliography{all,extra}

\end{document}

%% file: astdecls.tex
\usepackage{amssymb,amsmath,stmaryrd,mathrsfs,amsthm}
\usepackage[all]{xy}
\usepackage[neveradjust]{paralist}
\usepackage{color}
\definecolor{darkgreen}{rgb}{0,0.45,0} 
\usepackage[pagebackref,colorlinks,citecolor=darkgreen,linkcolor=darkgreen]{hyperref}
  \renewcommand*{\backref}[1]{}
  \renewcommand*{\backrefalt}[4]{({%
      \ifcase #1 Not cited.%
            \or On p.~#2%
            \else On pp.~#2%
      \fi%
    })}
\usepackage{mathtools}
\usepackage{wasysym}
\usepackage{braket}
\usepackage{version}
\let\setof\Set

\makeatletter

\def\mdef#1#2{\expandafter\expandafter\expandafter\gdef\expandafter\expandafter\noexpand#1\expandafter{\expandafter\ensuremath\expandafter{#2}}}

\newcommand{\cE}{\ensuremath{\mathcal{E}}}
\newcommand{\cM}{\ensuremath{\mathcal{M}}}

\newcommand{\cT}{\ensuremath{\mathcal{T}}}

\newcommand{\bbV}{\ensuremath{\mathbb{V}}}

\DeclareSymbolFont{bbold}{U}{bbold}{m}{n}
\DeclareSymbolFontAlphabet{\mathbbb}{bbold}
\mdef\bbSet{\mathbbb{Set}}

\newcommand{\bA}{\ensuremath{\mathbf{A}}}
\newcommand{\bB}{\ensuremath{\mathbf{B}}}
\newcommand{\bC}{\ensuremath{\mathbf{C}}}

\newcommand{\bS}{\ensuremath{\mathbf{S}}}
\newcommand{\bV}{\ensuremath{\mathbf{V}}}

\newcommand{\Xbar}{\ensuremath{\overline{X}}}
\newcommand{\Ybar}{\ensuremath{\overline{Y}}}

\newcommand{\inv}{^{-1}}
\newcommand{\meet}{\ensuremath{\wedge}}
\newcommand{\join}{\ensuremath{\vee}}

\SelectTips{cm}{}
\newdir{ >}{{}*!/-20pt/@{>}}

\newcommand{\pullbackcorner}[1][dr]{\save*!/#1-1.2pc/#1:(-1,1)@^{|-}\restore}

\mdef\eqv{\simeq}

\newcommand{\too}[1][]{\ensuremath{\overset{#1}{\longrightarrow}}}

\newcommand{\imp}{\ensuremath{\Rightarrow}}
\mdef\impp{\,\Longrightarrow\,}
\newcommand{\toto}{\ensuremath{\rightrightarrows}}
\newcommand{\into}{\ensuremath{\hookrightarrow}}
\newcommand{\mono}{\rightarrowtail}
\newcommand{\epi}[1][]{\ensuremath{\overset{#1}{\twoheadrightarrow}}}
\newcommand{\maps}{\colon}

\let\xto\xrightarrow

\def\rightarrowtailfill@{\arrowfill@{\Yright\joinrel\relbar}\relbar\rightarrow}
\newcommand\xrightarrowtail[2][]{\ext@arrow 0055{\rightarrowtailfill@}{#1}{#2}}
\let\xmono\xrightarrowtail

\def\twoheadrightarrowfill@{\arrowfill@{\relbar\joinrel\relbar}\relbar\twoheadrightarrow}
\newcommand\xtwoheadrightarrow[2][]{\ext@arrow 0055{\twoheadrightarrowfill@}{#1}{#2}}
\let\xepi\xtwoheadrightarrow

\def\twoheadleftarrowfill@{\arrowfill@\twoheadleftarrow\relbar{\relbar\joinrel\relbar}}
\newcommand\xtwoheadleftarrow[2][]{\ext@arrow 0055{\twoheadleftarrowfill@}{#1}{#2}}

\newtheorem{thm}{Theorem}[section]
  
\newtheorem{cor}{Corollary}
  \let\c@cor\c@thm
  \numberwithin{cor}{section}
  
\newtheorem{prop}{Proposition}
  \let\c@prop\c@thm
  \numberwithin{prop}{section}
  
\newtheorem{lem}{Lemma}
  \let\c@lem\c@thm
  \numberwithin{lem}{section}
  
\theoremstyle{definition}
\newtheorem{defn}{Definition}
  \let\c@defn\c@thm
  \numberwithin{defn}{section}
  
\newtheorem{defns}{Definitions}
  \let\c@defns\c@thm
  \numberwithin{defns}{section}
  
\newtheorem{notn}{Notation}
  \let\c@notn\c@thm
  \numberwithin{notn}{section}

  \let\c@notns\c@thm
  \numberwithin{notns}{section}
  
\newtheorem{cnv}{Convention}
  \let\c@cnv\c@thm
  \numberwithin{cnv}{section}
  
\theoremstyle{remark}
\newtheorem{rmk}{Remark}
  \let\c@rmk\c@thm
  \numberwithin{rmk}{section}

  \let\c@rmks\c@thm
  \numberwithin{rmks}{section}

  \let\c@eg\c@thm
  \numberwithin{eg}{section}

  \let\c@egs\c@thm
  \numberwithin{egs}{section}

\def\thmqedhere{\expandafter\csname\csname @currenvir\endcsname @qed\endcsname}

\let\c@equation\c@thm
\numberwithin{equation}{section}

\newlength\oldleftmargini       
\newlength\oldleftmarginii
\newlength\oldleftmarginiii
\newlength\oldleftmarginiv
\newlength\oldleftmarginv
\newlength\oldleftmarginvi
\newcount\maxenum
\newif\ifkillspacing
\def\@adjust@enum@labelwidth{%
  \advance\@listdepth by 1\relax
  \ifkillspacing                
    \csname c@\@enumctr\endcsname\maxenum
    \settowidth{\@tempdima}{%
      \csname label\@enumctr\endcsname\hspace{\labelsep}}%
    \csname leftmargin\romannumeral\@listdepth\endcsname
      \@tempdima
  \else                         
    \csname fixspacing\romannumeral\@listdepth\endcsname
  \fi
  \advance\@listdepth by -1\relax}
\def\fixspacingi{\ifnum\oldleftmargini=0\setlength\oldleftmargini\leftmargini\else\setlength\leftmargini\oldleftmargini\fi}
\def\fixspacingii{\ifnum\oldleftmarginii=0\setlength\oldleftmarginii\leftmarginii\else\setlength\leftmarginii\oldleftmarginii\fi}
\def\fixspacingiii{\ifnum\oldleftmarginiii=0\setlength\oldleftmarginiii\leftmarginiii\else\setlength\leftmarginiii\oldleftmarginiii\fi}
\def\fixspacingiv{\ifnum\oldleftmarginiv=0\setlength\oldleftmarginiv\leftmarginiv\else\setlength\leftmarginiv\oldleftmarginiv\fi}
\def\fixspacingv{\ifnum\oldleftmarginv=0\setlength\oldleftmarginv\leftmarginv\else\setlength\leftmarginv\oldleftmarginv\fi}
\def\fixspacingvi{\ifnum\oldleftmarginvi=0\setlength\oldleftmarginvi\leftmarginvi\else\setlength\leftmarginvi\oldleftmarginvi\fi}

\def\pl@label#1#2{%
  \edef\pl@the{\noexpand#1{\@enumctr}}%
  \pl@lab\expandafter{\the\pl@lab\csname yourthe\@enumctr\endcsname}%
  \advance\@tempcnta1
  \pl@loop}
\def\@enumlabel@#1[#2]{%
  \@plmylabeltrue
  \@tempcnta0
  \pl@lab{}%
  \let\pl@the\pl@qmark
  \expandafter\pl@loop\@gobble#2\@@@
  \ifnum\@tempcnta=1\else
    \PackageWarning{paralist}{Incorrect label; no or multiple
      counters.\MessageBreak The label is: \@gobble#2}%
  \fi
  \expandafter\edef\csname label\@enumctr\endcsname{\the\pl@lab}%
  \expandafter\edef\csname the\@enumctr\endcsname{\the\pl@lab}%
  \expandafter\let\csname yourthe\@enumctr\endcsname\pl@the
  #1}

\def\alwaysmath#1{\expandafter\expandafter\expandafter\global\expandafter\expandafter\expandafter\let\expandafter\expandafter\csname your@#1\endcsname\csname #1\endcsname
  \expandafter\def\csname #1\endcsname{\ensuremath{\csname your@#1\endcsname}}}
\alwaysmath{alpha}
\alwaysmath{beta}
\alwaysmath{gamma}
\alwaysmath{Gamma}
\alwaysmath{delta}
\alwaysmath{Delta}
\alwaysmath{epsilon}

\alwaysmath{zeta}
\alwaysmath{eta}
\alwaysmath{vartheta}
\alwaysmath{theta}
\alwaysmath{Theta}
\alwaysmath{iota}
\alwaysmath{kappa}
\alwaysmath{lambda}
\alwaysmath{Lambda}
\alwaysmath{mu}
\alwaysmath{nu}
\alwaysmath{xi}
\alwaysmath{Xi}
\alwaysmath{pi}
\alwaysmath{rho}
\alwaysmath{sigma}
\alwaysmath{Sigma}
\alwaysmath{tau}
\alwaysmath{upsilon}
\alwaysmath{Upsilon}
\alwaysmath{phi}
\alwaysmath{Pi}
\alwaysmath{Phi}
\newcommand{\ph}{\ensuremath{\varphi}}
\alwaysmath{chi}
\alwaysmath{psi}
\alwaysmath{Psi}
\alwaysmath{omega}
\alwaysmath{Omega}
\alwaysmath{ell}
\alwaysmath{infty}
\alwaysmath{odot}
\alwaysmath{top}
\alwaysmath{bot}
\alwaysmath{neg}
\mdef\na{\nabla}

\let\Gm\Gamma

\let\Th\Theta
\let\vth\vartheta

\mdef\ff{\Vdash}
\mdef\fff{\Vvdash}
\mdef\ss{\vDash}
\mdef\pp{\mathrel{\,\vdash\,}}
\mdef\vthtil{\widetilde{\vth}}
\mdef\Set{\mathbf{Set}}
\mdef\ffs{\ff_\bS}
\mdef\ffc{\ff_\bC}
\def\nno{\textsc{nno}}
\def\apg{\textsc{apg}}
\def\apgs{\textsc{apg}s}
\let\Iff\Leftrightarrow
\def\qq#1{$\ulcorner${\sf#1}$\urcorner$}

\alwaysmath{prec}
\alwaysmath{preceq}
\alwaysmath{star}
\mdef\St{\mathcal{S}\mathit{t}}
\mdef\Cat{\mathcal{C}\mathit{at}}
\mdef\iin{\mathrel{\epsilon}}
\mdef\ordiin{{\mathord{\epsilon}}}

\let\yourexists\exists
\def\exists#1.{(\yourexists#1)}
\let\yourforall\forall
\def\forall#1.{(\yourforall#1)}
\let\im\yourexists
\let\coim\yourforall
\alwaysmath{im}
\alwaysmath{coim}

\newcommand{\mm}[1]{\ensuremath{\llbracket#1\rrbracket}}

\DeclareMathOperator\Sub{Sub}

\mdef\bSh{\mathbf{Sh}}

\mdef\phhat{\widehat{\ph}}
\mdef\psihat{\widehat{\psi}}
\mdef\phtil{\widetilde{\ph}}
\mdef\psitil{\widetilde{\psi}}

\def\toiso{\xto{\smash{\raisebox{-.5mm}{$\scriptstyle\sim$}}}}

\mdef\ordin{\mathord{\in}}

\mdef\cin{\mathrel{\colon\hspace{-1mm}\mathord{\in}}}
\mdef\csub{\mathrel{\colon\hspace{-1mm}\mathord{\subseteq}}}

\newenvironment{blist}{\begin{list}{\labelitemi}{\leftmargin=1.3em\labelwidth=1em}}{\end{list}}
\def\setcurlabel#1{\def\@currentlabel{#1}}

\def\csl{\!\sslash\!}
\mdef\ddo{\Delta_0}

\def\pair(#1,#2){\langle #1,#2\rangle}

\makeatother